\documentclass{article}
\usepackage{geometry}
\usepackage{amssymb}
\usepackage{amsmath}
\usepackage{amsthm}
\usepackage{tikz}
\usetikzlibrary{calc}
\usepackage{tikz-cd}
\usepackage{footnotehyper}
\usepackage[colorlinks,linkcolor=blue!50!green]{hyperref}
\usepackage{pdfcomment}
\usepackage{lastpage}
\usepackage{mathtools}
\usepackage{braket}
\usepackage{aliascnt}
\usepackage[nameinlink,capitalize]{cleveref}
\usepackage{caption}
\usepackage{xspace}

\DeclareFontFamily{U}{wncy}{}
    \DeclareFontShape{U}{wncy}{m}{n}{<->wncyr10}{}
    \DeclareSymbolFont{mcy}{U}{wncy}{m}{n}
    \DeclareMathSymbol{\Sha}{\mathord}{mcy}{"58} 

\newcounter{theorem}[section]
\renewcommand{\thetheorem}{{\thesection.\arabic{theorem}}}

\theoremstyle{definition}

\theoremstyle{remark}
\newtheorem{remark}[theorem]{Remark}

\usepackage[framemethod=TikZ]{mdframed}

\newcounter{problem}
\mdfdefinestyle{problemstyle}{%
	middlelinewidth=3pt, outerlinewidth=0pt,
	middlelinecolor=green!55!black, roundcorner=5pt,
	backgroundcolor=green!5,
	settings={\global\refstepcounter{problem}},
}

\mdfdefinestyle{defstyle}{%
	middlelinewidth=0pt, outerlinewidth=0pt,
	middlelinecolor=blue!55!black, roundcorner=0pt,
	backgroundcolor=blue!5,
	settings={\global\refstepcounter{definition}},
}

\mdfdefinestyle{exstyle}{%
	middlelinewidth=0pt, outerlinewidth=0pt,
	roundcorner=0pt,
	backgroundcolor=orange!5,
	settings={\global\refstepcounter{example}},
}

\mdfdefinestyle{asidestyle}{%
	middlelinewidth=0pt, outerlinewidth=0pt,
	roundcorner=0pt,
	backgroundcolor=gray!5,
	footnoteinside=false,
}

\mdfdefinestyle{thmstyle}{%
	middlelinewidth=0pt, outerlinewidth=0pt,
	middlelinecolor=magenta!55!black, roundcorner=0pt,
	backgroundcolor=magenta!5,
}

\mdfdefinestyle{propstyle}{%
	middlelinewidth=0pt, outerlinewidth=0pt,
}

\mdfdefinestyle{proofstyle}{%
	middlelinewidth=0pt, outerlinewidth=0pt,
	roundcorner=5pt,
	backgroundcolor=magenta!3,
}

\newenvironment{aside}%
{\begin{mdframed}[style=asidestyle]}%
	{\end{mdframed}}

\usepackage{xparse}
\ExplSyntaxOn
\NewDocumentEnvironment{proofshift}{s}%
{\IfBooleanF{#1}{\vspace{-.7em}}\begin{mdframed}[style=proofstyle]\begin{proof}}%
		{\end{proof}\end{mdframed}}

\NewDocumentEnvironment{problem}{o o o}
{
	\begin{mdframed}[style=problemstyle, 
		\IfValueT{#3}{#3,}
		frametitlefont=\normalfont,
		frametitle={\textbf{
				\IfNoValueTF{#2}
				{Problem~\theproblem.}
				{\pdfmarkupcomment[opacity=0,color=green!5]{\textbf Problem~\theproblem.}{#2}}
			}
			\IfValueT{#1}{\hfill #1}}
		]
	}
	{\end{mdframed}
}
\newaliascnt{definition}{theorem}
\NewDocumentEnvironment{definition}{o}
{
	\begin{mdframed}[style=defstyle, frametitle={
			\IfNoValueTF{#1}
			{Definition~\thedefinition.}
			{\pdfmarkupcomment[opacity=0]{Definition~\thedefinition.}{#1}}
		}]
	}
	{\end{mdframed}
}
\newaliascnt{example}{theorem}
\NewDocumentEnvironment{example}{o}
{
	\begin{mdframed}[style=exstyle, frametitle={
			\IfNoValueTF{#1}
			{Example~\theexample.}
			{Example~\theexample :~ {#1}.}
		}]
	}
	{\end{mdframed}
}
\NewDocumentEnvironment{theorem}{s o} 
{
    \refstepcounter{theorem}
	\begin{mdframed}[style=thmstyle, frametitle={
			\IfNoValueTF{#2}
			{Theorem~\thetheorem.}
			{Theorem~\thetheorem :~ {#2}}
		},
		beforesingleframe={\IfBooleanT{#1}
			{
				\tikz[remember~picture, overlay]{\fill[fill=magenta!3] 
					(0,10pt) rectangle (\textwidth,-10pt);}  
		}}
		]
	}
	{\end{mdframed}}
\newaliascnt{lemma}{theorem}
\NewDocumentEnvironment{lemma}{s o} 
{
    \refstepcounter{lemma}
	\begin{mdframed}[style=thmstyle, frametitle={
			\IfNoValueTF{#2}
			{Lemma~\thelemma.}
			{Lemma~\thelemma :~ {#2}.}
		},
		beforesingleframe={\IfBooleanT{#1}
			{
				\tikz[remember~picture, overlay]{\fill[fill=magenta!3] 
					(0,10pt) rectangle (\textwidth,-10pt);}  
		}}
		]
	}
	{\end{mdframed}}
\newaliascnt{proposition}{theorem}
\NewDocumentEnvironment{proposition}{s o} 
{
    \refstepcounter{proposition}
	\begin{mdframed}[style=thmstyle, frametitle={
			\IfNoValueTF{#2}
			{Proposition~\theproposition.}
			{Proposition~\theproposition :~ {#2}.}
		},
		]
	}
	{\end{mdframed}}
\newaliascnt{corollary}{theorem}
\NewDocumentEnvironment{corollary}{s o} 
{
    \refstepcounter{corollary}
	\begin{mdframed}[style=thmstyle, frametitle={
			\IfNoValueTF{#2}
			{Corollary~\thecorollary.}
			{Corollary~\thecorollary :~ {#2}.}
		},
		beforesingleframe={\IfBooleanT{#1}
			{
				\tikz[remember~picture, overlay]{\fill[fill=magenta!3] 
					(0,10pt) rectangle (\textwidth,-10pt);}  
		}}
		]
	}
	{\end{mdframed}}
\newaliascnt{conjecture}{theorem}
\NewDocumentEnvironment{conjecture}{s o} 
{
    \refstepcounter{conjecture}
	\begin{mdframed}[style=thmstyle, frametitle={
			\IfNoValueTF{#2}
			{Conjecture~\theconjecture.}
			{Conjecture~\theconjecture :~ {#2}.}
		}
		]
	}
	{\end{mdframed}}
\ExplSyntaxOff

\newcommand{\probbox}[2][]{\tikz[overlay]\node[draw=green!55!black,fill=green!5,inner sep=2pt, anchor=text, rectangle, rounded corners=1mm,#1] {#2};\phantom{#2}}
\newcommand{\pref}[1]{\probbox{\hyperref[#1]{Problem \ref*{#1}}}}

\newcommand{\defbox}[2][]{\tikz[overlay]\node[draw=blue,fill=blue!5,inner sep=2pt, anchor=text, rectangle, rounded corners=1mm,#1] {#2};\phantom{#2}}
\newcommand{\dref}[1]{\defbox{\hyperref[#1]{Definition \ref*{#1}}}}

\newcommand{\thmbox}[2][]{\tikz[overlay]\node[draw=magenta!55!black,fill=magenta!3,inner sep=2pt, anchor=text, rectangle, rounded corners=1mm,#1] {#2};\phantom{#2}}
\NewDocumentCommand{\tref}{m o}{%
	\thmbox{%
		\hyperref[#1]{\IfValueTF{#2}{#2}{Theorem \ref*{#1}}}%
	}%
}

\newcommand{\newword}[1]{\textbf{#1}}

\newcommand{\NN}{\mathbb{N}}

\newcommand{\QQ}{\mathbb{Q}}
\newcommand{\RR}{\mathbb{R}}

\newcommand{\ZZ}{\mathbb{Z}}

\newcommand{\cA}{\mathcal{A}}
\newcommand{\cB}{\mathcal{B}}
\newcommand{\cC}{\mathcal{C}}

\newcommand{\cL}{\mathcal{L}}

\newcommand{\cR}{\mathcal{R}}
\newcommand{\cS}{\mathcal{S}}

\newcommand{\tA}{\widetilde{A}}
\newcommand{\tB}{\widetilde{B}}
\newcommand{\tC}{\widetilde{C}}
\newcommand{\tD}{\widetilde{D}}

\newcommand{\tS}{\widetilde{S}}
\newcommand{\tT}{\widetilde{T}}

\newcommand{\tpi}{\widetilde{\pi}}

\newcommand{\tN}{\widetilde{N}}
\newcommand{\talpha}{\widetilde{\alpha}}

\newcommand{\be}{\mathbf{e}}
\newcommand{\tbe}{\widetilde{\mathbf{e}}}

\renewcommand{\aa}{{\mathfrak{a}}} 
\newcommand{\bb}{{\mathfrak{b}}}
\newcommand{\taa}{{\widetilde{\mathfrak{a}}}} 

\newcommand{\rind}[1]{\mathopen{\boldsymbol{(}}#1\mathclose{\boldsymbol{)}}}

\newcommand{\LowerArcs}{\mathsf{LowerArcs}}
\newcommand{\UpperArcs}{\mathsf{UpperArcs}}
\newcommand{\tLowerArcs}{\widetilde{\mathsf{LowerArcs}}}

\newcommand{\LowerWalls}{\mathsf{LowerWalls}}

\newcommand{\JIrrc}{\mathsf{JIrr}^c}
\newcommand{\MIrrc}{\mathsf{MIrr}^c}
\newcommand{\JI}{cJI\xspace}
\newcommand{\JIs}{cJIs\xspace}
\newcommand{\Arcs}{\mathsf{Arcs}}

\newcommand{\Tot}{\mathrm{Tot}}
\newcommand{\TTot}{\mathrm{TTot}}
\newcommand{\WO}{\mathsf{WO}}

\newcommand{\Dyer}{\mathsf{Dyer}}

\newcommand{\precdot}{\prec\mathrel{\mkern-5mu}\mathrel{\cdot}}

\title{Extended weak order for the affine symmetric group}
\author{Grant T. Barkley}
\date{}

\begin{document}

\maketitle

\begin{abstract}
    The extended weak order on a Coxeter group $W$ is the poset of biclosed sets in its root system. In \cite{Barkley2022}, it was shown that when $W=\widetilde{S}_n$ is the affine symmetric group, then the extended weak order is a quotient of the lattice $L_n$ of translation-invariant total orderings of the integers. In this article, we give a combinatorial introduction to $L_n$ and the extended weak order on $\widetilde{S}_n$. We show that $L_n$ is an algebraic completely semidistributive lattice. We describe its canonical join representations using a cyclic version of Reading's non-crossing arc diagrams. We also show analogous statements for the lattice of all total orders of the integers, which is the extended weak order on the symmetric group $S_\infty$. A key property of both of these lattices is that they are profinite; we also prove that a profinite lattice is join semidistributive if and only if its compact elements have canonical join representations. We conjecture that the extended weak order of any Coxeter group is a profinite semidistributive lattice.
\end{abstract}

\tableofcontents

\section{Introduction}

The \newword{extended weak order} of a Coxeter group $W$ was introduced by Matthew Dyer in the study of reflection orders and Kazhdan--Lusztig polynomials \cite{Dyer1993,Dyer2019}. We denote the extended weak order of $W$ by $\Dyer(W)$. It is the poset of \newword{biclosed sets} in the root system of $W$, ordered by containment. $\Dyer(W)$ includes as subposets the \newword{weak order} on $W$ (which is the partial order on $W$ given by orienting its Cayley graph away from the identity) and the \newword{limit weak order} on $W$ \cite{Lam2013,Wang2019} (which includes words of infinite length in the generators of $W$). Extended weak order has been the subject of many conjectures since its introduction, most of which are extensions of properties that are known to hold for the weak order on $W$. For instance, it is conjectured that $\Dyer(W)$ is a complete lattice \cite{Dyer2019}. When $W$ is finite, the conjecture is equivalent to an observation of Bj\"orner \cite{BjornerLatticeFinite} that weak order on a finite Coxeter group is a lattice. The conjecture was solved for rank 3 affine Coxeter groups in \cite{Wang2019} and for all affine Coxeter groups in \cite{Barkley2023}. All other cases are open. Recent work has also made conjectural connections between extended weak order and the lattice of torsion classes for preprojective algebras \cite{Dana2023}.

One goal of this article is to give an introduction to the combinatorics of the extended weak order when $W$ is the affine symmetric group $\tS_n$, which is the affine Coxeter group of type $\tA_{n-1}$. In this case, $\Dyer(\tS_n)$ was shown to be a lattice in \cite{Barkley2022,Barkley2023} (and was earlier known to be a lattice by Dyer). There are two other lattices closely related to $\Dyer(W)$. One is the \newword{weak order on total orders} of the integers, denoted $\WO(\Tot)$. It was shown in \cite{Barkley2022} that $\WO(\Tot)$ is isomorphic to $\Dyer(S_\infty)$, the extended weak order on the infinite symmetric group $S_\infty$. The other lattice is the \newword{weak order on translation-invariant total orders} (TITOs), denoted $\WO(\TTot_n)$. It was shown in \cite{Barkley2022} that $\Dyer(\tS_n)$ is a lattice quotient of $\WO(\TTot_n)$\footnote{The relationship between $\Dyer(S_\infty)$ and $\Dyer(\tS_n)$ is that of \emph{folding} the root system of $S_\infty$ to the root system of $\tS_n$.}. Similar descriptions of $\Dyer(W)$ for $W$ of type $\tB,$ $\tC,$ and $\tD$ are given in \cite[Section 6]{Barkley2022}. It turns out that many properties of $\Dyer(\tS_n)$ follow immediately from properties of $\WO(\TTot_n)$. As a result, we shall focus in this article on the lattices $\WO(\Tot)$ and $\WO(\TTot_n)$, 
and deduce the consequences for $\Dyer(\tS_n)$ as corollaries. Our main theorem is the following.

\begin{theorem}\label{thm:intro}
    The lattices $\WO(\Tot)$, $\WO(\TTot_n)$, and $\Dyer(\tS_n)$ are profinite semidistributive lattices. 
\end{theorem}

\newword{Semidistributivity} is an important property of the weak order on a finite Coxeter group, and is the foundation for constructions such as canonical join representations \cite{Freese}, shard intersection orders \cite{ReadingShardInt}, and rowmotion \cite{RowSlow}. The notion of a \newword{profinite} lattice (see, for instance, \cite{Vosmaer2010PhD}) seems to be relatively unstudied outside of universal algebra. However, we demonstrate here that it is a powerful tool for studying the extended weak order. For instance, all profinite lattices are \newword{algebraic lattices} (\Cref{cor:algebraic}), which is a more well-studied type of lattice that is determined by its \newword{compact} elements. Furthermore, a profinite semidistributive lattice is also \newword{completely semidistributive}. We make the following conjecture generalizing \Cref{thm:intro}:

\begin{conjecture}\label{conj:intro}
    If $W$ is any Coxeter group, then $\Dyer(W)$ is a profinite semidistributive lattice.
\end{conjecture}
We note that the ``profinite'' part of Conjecture \ref{conj:intro} follows quickly for affine Coxeter groups from the results in \cite{Barkley2023}.

To study the lattices $\WO(\Tot)$ and $\WO(\TTot_n)$, we introduce \newword{arc diagrams} which describe some of their elements. Arc diagrams were introduced by Nathan Reading in \cite{Reading2015} as a way to encode the \newword{canonical join representation} of an element of the finite symmetric group $S_n$. Our arc diagrams for $\WO(\Tot)$  
are an obvious generalization of Reading's definition, and our \newword{cyclic arc diagrams} for $\WO(\TTot_n)$ can be identified with the translation-invariant arc diagrams of $\WO(\Tot)$. In $S_n$, every element has a canonical join representation, so the arc diagram of a permutation uniquely determines it. In a general infinite completely semidistributive lattice, including $\WO(\Tot)$ and $\WO(\TTot_n)$, this is no longer the case. We introduce the notion of a \newword{widely generated} element of a poset (\Cref{def:widelygenerated}). The widely generated elements of a profinite completely semidistributive lattice turn out to be exactly the ones that have a canonical join representation. We study this notion in depth for $\WO(\Tot)$ and $\WO(\TTot_n)$. 

This paper is written as an introduction to these lattices for the combinatorialist. We do not require any knowledge of Coxeter groups and have included an extensive background section detailing the lattice theory of the weak order on $S_n$. In \Cref{sec:Sinfty} we study $\WO(\Tot)$ and introduce the notion of a widely generated element of a poset. In \Cref{sec:TITO} we study $\WO(\TTot_n)$ and its cyclic arc diagrams. We conclude in \Cref{sec:profinite} by discussing properties of profinite lattices and proving \Cref{thm:intro}. 

\section{Background} 
\label{sec:symmetric}
We begin by recalling the combinatorics of the weak order on the symmetric group.
Let $S_n$ denote the group of permutations of the set $\{1,\ldots,n\}$. 
\begin{definition}\label{def:weakorderSn}
We say that the pair $(a,b)$ is an \newword{inversion} of $\pi$ if 
\( a < b \text{ and } \pi^{-1}(b) < \pi^{-1}(a). \)
Write $N(\pi)$ for the set of inversions of $\pi$. The \newword{weak order} on $S_n$ is the partial ordering such that $\pi_1 \leq \pi_2$ if and only if $N(\pi_1)\subseteq N(\pi_2)$.    
\end{definition} 
Implicit in \Cref{def:weakorderSn} is the fact that the set $N(\pi)$ determines $\pi$ uniquely.  If we write a permutation in one-line notation, then the inversions are the pairs which are out of order. 

\begin{example}
The inversions of the permutation $51423$ are
\[(1,5),~(4,5),~(2,5),~(3,5),~(2,4),~(3,4).\]
\end{example}
\Cref{fig:S3weak} depicts the Hasse diagram of weak order on $S_3$. 

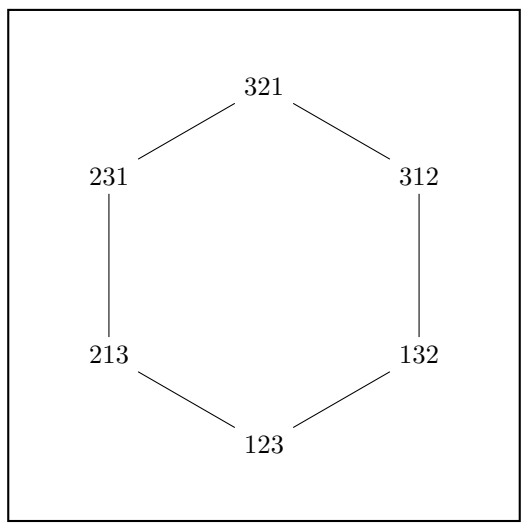
\begin{figure}
    \centering
    \begin{tikzpicture}[scale=1.7]
        \node (B) at (270:1.4) {$123$};
        \node (L1) at (210:1.4) {$213$};
        \node (L2) at (150:1.4) {$231$};
        \node (T) at (90:1.4) {$321$};
        \node (R2) at (30:1.4) {$312$};
        \node (R1) at (-30:1.4) {$132$};
        \draw (B) -- (L1) -- (L2) -- (T) -- (R2) -- (R1) -- (B);
        \draw[thick] (-2,-2) rectangle (2,2);
    \end{tikzpicture}
    \caption{The Hasse diagram of weak order on $S_3$.}
    \label{fig:S3weak}
\end{figure}

\subsection{Lattice structure on $S_n$}

\begin{definition}
    Let $P$ be a poset. Then we say that $P$ is a (bounded) \newword{lattice} if every finite subset $X\subseteq P$ has a least upper bound, called its \newword{join}, and a greatest lower bound, called its \newword{meet}. We denote the join by $\bigvee X$ or $\bigvee_{x\in X} x$ and the meet by $\bigwedge X$ or $\bigwedge_{x\in X} x$. If \emph{every} subset $X\subseteq P$ has a join and a meet, then we say that $P$ is a \newword{complete lattice}.
\end{definition}

Every finite lattice is a complete lattice, but for infinite posets, being a complete lattice is a stronger property. The following is well-known. 
\begin{proposition}\label{prop:latticeSn}
    The weak order on $S_n$ is a (complete) lattice.
\end{proposition}
To compute the join of a list of permutations, we will use the method of \cite[Theorem 10-3.25]{Reading2016b} and introduce a closure operator on sets of inversions. Write $T\coloneqq \{(a,b)\mid 1\leq a < b \leq n\}.$ If $N\subseteq T$, then we define the \newword{closure} of $N$ to be the minimal set $\overline{N}\supseteq N$ such that if $a<b<c$ and $(a,b)$, $(b,c)$ are both in $\overline{N}$, then $(a,c)$ is in $\overline{N}$. The join of two permutations $\pi_1$ and $\pi_2$, denoted $\pi_1\vee \pi_2$, is the unique permutation with inversion set $\overline{N(\pi_1)\cup N(\pi_2)}.$
More generally, the join of a family $\{\pi_i\}_{i\in I}$ has inversion set $\overline{\bigcup_{i\in I}N(\pi_i)}$. One can compute the meet of a family $\{\pi_i\}_{i\in I}$ dually: it has inversion set $(\bigcap_{i\in I}N(\pi_i))^\circ$, where $N^\circ\coloneqq T\setminus \overline{T\setminus N}$ is the \newword{interior} of $N$.

\begin{example}
Let's compute the join of $213$ and $132$. Then $N(213)=\{(1,2)\}$ and $N(132)=\{(2,3)\}$. We need to compute the closure of $N(213)\cup N(132) = \{(1,2),(2,3)\}$. The closure is forced to contain $(1,3)$, since $(1,2)$ and $(2,3)$ are both elements. Hence 
$\overline{N(213)\cup N(132)}=\{(1,2),(2,3),(1,3)\}.$
The join $213\vee 132$ should be the unique permutation with this inversion set, which is the permutation $321$. As can be seen in \Cref{fig:S3weak}, we indeed have $213\vee 132=321$.

\end{example}

\subsection{Arc diagrams}
In this section, we summarize Nathan Reading's description of canonical join representations in $S_n$ via non-crossing arc diagrams \cite{Reading2015}. We use the language from \cite{Padrol}.
Fix $n\in \NN$. Then an \newword{arc} for $S_n$ is a tuple $(a,b,L,R)$, where $1 \leq a < b \leq n$ and $L,R$ is a partition of the integers $\{a+1, a+2, \ldots, b-1\}$ into two sets: the \newword{left set} $L$ and the \newword{right set} $R$. We depict arcs using \newword{arc diagrams}. The arc $(a,b,L,R)$ is drawn as an arc starting at $a$ (its \newword{initial value}) and ending at $b$ (its \newword{terminal value}), passing over the elements of $L$ and under the elements of $R$. Examples are shown in \Cref{fig:shardarc}.
\begin{definition}
Let $\pi$ be a permutation. A pair $(a,b)$ is called a \newword{lower wall} of $\pi$ if $\pi^{-1}(a) = \pi^{-1}(b)+1$. Similarly, $(a,b)$ is an \newword{upper wall} if $\pi^{-1}(a) = \pi^{-1}(b)-1$. An arc $(a,b,L,R)$ is said to be a \newword{lower arc} (respectively, \newword{upper arc}) of $\pi$ if

\begin{itemize}
    \item $(a,b)$ is a lower (respectively, upper) wall of $\pi$, and
    \item $\pi^{-1}(\ell) < \pi^{-1}(a)$ for all $\ell\in L$, and
    \item $\pi^{-1}(b) < \pi^{-1}(r)$ for all $r\in R$.
\end{itemize} 
\end{definition}

We identify the pair $(a,b)$ with the transposition swapping $a$ and $b$. Hence $(a,b)\cdot \pi$ denotes product of two permutations. For two elements $x,y$ in a poset, we say $y$ \newword{covers} $x$ if $x<y$ and there does not exist $z$ so that $x<z<y$. We write $x\lessdot y$ to mean that $y$ covers $x$. We note that if $\pi' \leq \pi$ and $|N(\pi')| = |N(\pi)|-1$, then necessarily $\pi'\lessdot \pi$. The following shows that these are the only cover relations in weak order, and relates covers to lower walls.

\begin{proposition}\label{prop:wallscoverSn}
    The following are equivalent, for a permutation $\pi$ and a pair $(a,b)$ with $a<b$:
    \begin{itemize}
        \item[(a)] The pair $(a,b)$ is a lower wall of $\pi$;
        \item[(b)] The product $(a,b)\cdot \pi$ is covered by $\pi$;
        \item[(c)] The pair $(a,b)$ is in $N(\pi)$ and $N(\pi)\setminus \{(a,b)\}$ is the inversion set of some permutation.
    \end{itemize}
    Conversely, if $\pi' \lessdot \pi$ is a cover relation, then there exists a unique pair $(a,b)$ so that $N(\pi) \setminus \{(a,b)\} = N(\pi')$.
\end{proposition}

\begin{example}
    In the permutation $\pi=1423$, the pair $(2,4)$ is a lower wall. The product $(2,4)\cdot \pi$ is $1243$, which is covered by $1423$ in weak order. The inversion sets are
    \begin{align*} N(1423) &= \{(2,4),(3,4)\} \\ N(1243) &= \{(3,4)\}. \end{align*}
\end{example}

\begin{example}
    Let $\pi$ be the permutation $25143$. Then the lower walls are $(1,5),(3,4)$ and the upper walls are $(2,5),(1,4)$. The lower arcs are $(1,5,\{2\},\{3,4\})$, $(3,4,\varnothing,\varnothing)$, and the upper arcs are $(2,5,\varnothing,\{3,4\})$, $(1,4,\{2\}, \{3\})$. These arcs are depicted below, with the lower arcs in black and the upper arcs in blue.
    
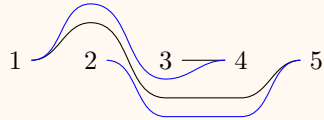
\begin{center}
        \begin{tikzpicture}
        \node (A1) at (1,0) {$1$};
        \node (A2) at (2,0) {$2$};
        \node (A3) at (3,0) {$3$};
        \node (A4) at (4,0) {$4$};
        \node (A5) at (5,0) {$5$};

        \draw (A1.east) to[out=0,in=180] (2,.5) to[out=0,in=180] (3,-.5) to[out=0,in=180] (4,-.5) to[out=0,in=180] (A5.west);
        \draw (A3.east) to (A4.west);
        \draw[blue] (A2.east) to[out=0,in=180] (3,-.75) to (4,-.75) to[out=0,in=180] (A5.west);
        \draw[blue] (A1.east) to[out=0,in=180] (2,.75) to[out=0,in=180] (3,-.25) to[out=0,in=180] (A4.west);
        \end{tikzpicture}
    \end{center}
\end{example}

The collection of lower arcs of a permutation is very well-behaved. We say the arcs $(a,b,L,R)$ and $(a',b',L',R')$ \newword{cross} if the depictions of the arcs in an arc diagram share initial values, share terminal values, or intersect in their relative interiors. More precisely, this occurs if and only if either $a=a'$, $b=b'$, or $(L\cap R') \cup (\{a,b\} \cap R') \cup (L\cap \{a',b'\})$ and $(R\cap L') \cup (\{a,b\} \cap L')\cup (R\cap \{a',b'\})$ are both nonempty. We say that a set of arcs $D$ is a \newword{non-crossing collection} if no two arcs in $D$ cross. 
In this case, a depiction of the arcs in $D$ in a single diagram is called a \newword{non-crossing arc diagram}.

Write $\LowerArcs(\pi)$ and $\UpperArcs(\pi)$ for the set of lower arcs of $\pi$ and the set of upper arcs of $\pi$, respectively.

\begin{proposition}\label{prop:noncrossingSn}
    For any permutation $\pi$, the set $\LowerArcs(\pi)$ is a non-crossing collection. Furthermore, any non-crossing collection $D$ is of the form $\LowerArcs(\pi)$ for a unique permutation $\pi$.
\end{proposition}

In other words, taking lower arcs gives a bijection between permutations and non-crossing collections. A similar statement holds for the upper arcs\footnote{In this article, we shall focus on lower arcs (and joins), with the understanding that every statement about lower arcs has a dual version for upper arcs (and meets).}.

\begin{definition}
    An element $x$ of a lattice $L$ is called \newword{join-irreducible} if there is no finite set $X\subseteq L$ with $x\not\in L$ so that $\bigvee X = x$. An element $x$ of a complete lattice $L$ is called a \newword{complete join-irreducible} (\JI) if there is no set $X\subseteq L$ with $x\not\in X$ so that $\bigvee X = x$. We write $\JIrrc(L)$ for the set of complete join-irreducibles in $L$. We similarly define complete meet-irreducibles and write $\MIrrc(L)$ for their collection. 
\end{definition}

The two notions coincide for finite lattices. We will exclusively be interested in complete join-irreducibles; however we may omit the word ``complete'' in the finite lattice setting or for brevity. Given a \JI $j$, there is a unique element covered by $j$, which we always denote by $j_*$. Dually, if $m$ is a complete meet-irreducible, then we write $m^*$ for the unique element covering $m$. In a finite lattice like $S_n$, an element $\pi$ is join-irreducible if and only if $\pi$ covers a unique element $\pi_*$. (In an infinite lattice, there may be elements covering a unique element which are not complete join-irreducibles.) Hence one can verify by inspecting \Cref{fig:S3weak} that the \JIs in $S_3$ are $213,~132,~231,$ and $312$.

By \Cref{prop:wallscoverSn}, a permutation is join-irreducible if and only if $|\LowerArcs(\pi)|=1$. If $\aa$ is an arc, then let $j_{\aa}$ be the unique permutation with $\LowerArcs(j_\aa) = \{\aa\}$. To make this explicit, assume $\aa=(a,b,\{\ell_1,\ldots,\ell_i\},\{r_1,\ldots,r_j\})$ with $\ell_1<\cdots<\ell_i$ and $r_1<\cdots < r_j$. Then $j_{\aa}$ is the permutation with one-line notation
\[ 1,~2,~\ldots,~ a-1,~\ell_1,~\ell_2,~\ldots,~\ell_{i-1},~\ell_i,~ b, ~a, ~ r_1, ~r_2, \ldots, ~r_{j-1},~ r_j,~ b+1,~ b+2,~ \ldots,~ n. \]
We have shown:
\begin{proposition}
    The map $\aa\mapsto j_\aa$ is a bijection from arcs to $\JIrrc(S_n)$.
\end{proposition}
\Cref{fig:shardarc} depicts several arcs with their associated join-irreducible. 
In the next subsection, we shall build upon this result to describe the lattice-theoretic significance of non-crossing collections.

\begin{figure}
\hspace{-1.5em}
\begin{tikzpicture}[scale=1]
\begin{scope}[shift={(-4,0)}]
\draw[thick] (-.5, -1.5) rectangle (3.5,1.5);
    \node at (1.5,1) {$\aa=(1,4,\{2,3\},\varnothing)$};

    \node (A1) at (0,0) {$1$};
    \node (A2) at (1,0) {$2$};
    \node (A3) at (2,0) {$3$};
    \node (A4) at (3,0) {$4$};

    \draw (A1.east) to[out=0,in=180] (1,.5) to (2,.5) to[out=0,in=180] (A4.west);
    
    \node at (1.5,-1) {$j_\aa=2341$};
\end{scope}
\begin{scope}[shift={(0,0)}]
\draw[thick] (-.5, -1.5) rectangle (3.5,1.5);
    \node at (1.5,1) {$\aa=(1,4,\{3\},\{2\})$};
    
    \node (A1) at (0,0) {$1$};
    \node (A2) at (1,0) {$2$};
    \node (A3) at (2,0) {$3$};
    \node (A4) at (3,0) {$4$};

    \draw (A1.east) to[out=0,in=180] (1,-.5) to[out=0,in=180] (2,.5) to[out=0,in=180] (A4.west);
    \node at (1.5,-1) {$j_\aa=3412$};
\end{scope}
\begin{scope}[shift={(4,0)}]
\draw[thick] (-.5, -1.5) rectangle (3.5,1.5);
    \node at (1.5,1) {$\aa=(1,4,\{2\},\{3\})$};
    
    \node (A1) at (0,0) {$1$};
    \node (A2) at (1,0) {$2$};
    \node (A3) at (2,0) {$3$};
    \node (A4) at (3,0) {$4$};

    \draw (A1.east) to[out=0,in=180] (1,.5) to[out=0,in=180] (2,-.5) to[out=0,in=180] (A4.west);
    \node at (1.5,-1) {$j_\aa=2413$};
\end{scope}
\begin{scope}[shift={(8,0)}]
\draw[thick] (-.5, -1.5) rectangle (3.5,1.5);
    \node at (1.5,1) {$\aa=(1,4,\varnothing,\{2,3\})$};

    \node (A1) at (0,0) {$1$};
    \node (A2) at (1,0) {$2$};
    \node (A3) at (2,0) {$3$};
    \node (A4) at (3,0) {$4$};

    \draw (A1.east) to[out=0,in=180] (1,-.5) to[out=0,in=180] (2,-.5) to[out=0,in=180] (A4.west);
    \node at (1.5,-1) {$j_\aa=4123$};
\end{scope}
\end{tikzpicture}
\caption{The arc diagrams for arcs in $S_4$ with initial value $1$ and terminal value $4$. Below each diagram, we have indicated the associated JI in $S_4$.}
\label{fig:shardarc}
\end{figure}
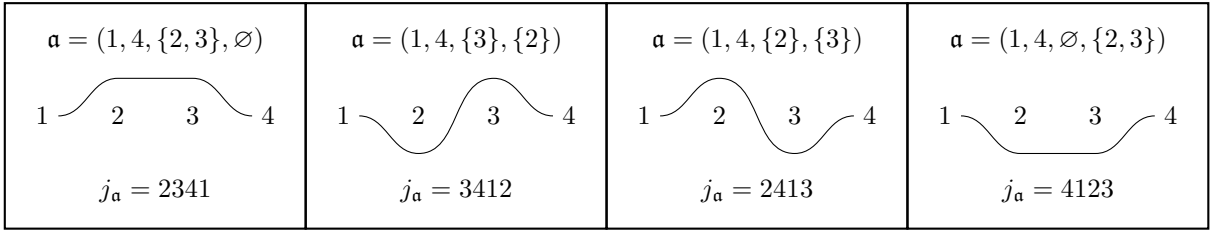

\subsection{Canonical join representations and semidistributivity}

\begin{definition}
    Let $L$ be a complete lattice and let $U,V\subseteq L$. We write $U\ll V$ to mean that for every $u\in U$, there is some $v\in V$ so that $u\leq v$. We say that $\bigvee U = x$ is an \newword{irredundant} join representation of $x$ if there is no proper subset $U'\subsetneq U$ so that $\bigvee U' = x$. We say that $\bigvee U = x$ is the \newword{canonical join representation} of $x$ if it is irredundant and for any $V\subseteq L$ so that $\bigvee V = x$, we have $U\ll V$.  
\end{definition}
For a general complete lattice $L$, there may be elements that do not have a canonical join representation. However, the following proposition (which is \cite[Proposition 2.4]{Reading2015}) shows that every element of $S_n$ has such a representation. 

\begin{proposition}\label{prop:canonicaljoinSn}
    Let $\pi\in S_n$. Then the canonical join representation of $\pi$ is 
    \[ \pi = \bigvee_{\aa\in \LowerArcs(\pi)} j_\aa . \]
\end{proposition}

Existence of canonical join representations in $L$ is related to a property called \emph{semidistributivity} of $L$.
\begin{definition}\label{def:semidistributive}
    A complete lattice $L$ is \newword{completely join semidistributive} if, for every $y,z\in L$, whenever $X\subseteq L$ satisfies $x\vee y = z$  for all $x\in X$, then also $(\bigwedge X) \vee y = z$. It is \newword{completely meet semidistributive} if, for every $y,z\in L$, whenever $X\subseteq L$ satisfies $x\wedge y = z$ for all $x\in X$, then also $(\bigvee X) \wedge y = z$. If $L$ is completely meet and join semidistributive, then we say $L$ is \newword{completely semidistributive}. If these conditions hold only when $X$ is a finite set, then we say $L$ is \newword{semidistributive}.
\end{definition}
If $L$ is finite, then the word ``completely'' is usually dropped. The definition of semidistributivity is \emph{a priori} not so interesting. For finite lattices, semidistributivity is accessed mostly via the following result \cite{Freese}.

\begin{proposition}\label{prop:finitecanonicaljoin}
    A finite lattice is (completely) join semidistributive if and only if every element has a canonical join representation.
\end{proposition}
\Cref{prop:canonicaljoinSn,prop:finitecanonicaljoin} imply that $S_n$ is a semidistributive lattice. This fact was known much earlier; it is originally a theorem of Duquenne and Cherfouh \cite{Duquenne}.
For general infinite lattices, there is no such equivalence between complete semidistributivity and canonical join representations. 
However, for \emph{profinite} lattices (see \Cref{def:profinite}), 
completely join semidistributive lattices do satisfy an analog of \Cref{prop:finitecanonicaljoin} (see \Cref{prop:profinitesemidist}).


    


\subsection{The poset of regions}
In this section, we will consider the relationship between the weak order and convex geometry. To see this, recall that a \newword{hyperplane arrangement} in $\RR^n$ is a collection of linear hyperplanes. 


\begin{definition}
    The \newword{braid arrangement} is the hyperplane arrangement 
    \[\cB_n \coloneqq \{H_{ab} \mid 1 \leq a < b \leq n\}, \] 
    where $H_{ab}$ is the hyperplane
\( H_{ab} \coloneqq \{ (x_1,\ldots,x_n)\in \RR^n \mid x_a = x_b \}.   \)
\end{definition}
In \Cref{fig:braid}, we've depicted (a slice through) $\cB_3$. As illustrated in the figure, two points are in the same \newword{region} (connected component of $\RR^n\setminus \bigcup_{a<b}H_{ab}$) if and only if their coordinates are in the same order. Hence regions correspond to total orderings of the coordinates $x_1,\ldots,x_n$. We can think of these total orderings as the one-line notation of a permutation, so that for instance the permutation $231$ corresponds to the region whose points have coordinates satisfying $x_2<x_3<x_1$. This gives a bijection between elements of $S_n$ and regions of $\cB_n$.

We can get this bijection in another way: there is a group action of $S_n$ on $\RR^n$, where $\pi$ acts via
\( (x_1,\ldots, x_n) \mapsto (x_{\pi^{-1}(1)},\ldots,x_{\pi^{-1}(n)}). \)
The action on points induces an action on the regions of $\cB_n$, and the action on regions is simply transitive. Hence if we fix a choice of \newword{base region} corresponding to the identity permutation, then the group action induces a bijection between regions and elements of $S_n$. If we use as a base region the region $B = \{(x_1,\ldots,x_n)\mid x_1<\cdots < x_n\}$, then the bijection is
\[ \pi \mapsto \pi B = \{(x_1,\ldots,x_n) \mid x_{\pi(1)} < \cdots < x_{\pi(n)}\}. \]

We can describe the weak order using an ordering on the regions of $\cB_n$. 
\begin{definition}
Let $\cA$ be a hyperplane arrangement in $\RR^n$. Given regions $R_1$ and $R_2$, their \newword{separating set} is
\[ \cS(R_1,R_2) \coloneqq \{H \in \cA \mid R_1\text{ and }R_2\text{ are in different components of } \RR^n\setminus H_{ab}\}.  \]
Fix a base region $B$. The partial order on regions putting $R_1 \leq R_2$ if $\cS(B,R_1)\subseteq \cS(B,R_2)$ is called the \newword{poset of regions} of $\cA$ (relative to $B$).
\end{definition}
Specializing this to $\cB_n$, we find that
\[ \cS(B,\pi B) = \{H_{ab} \mid (a,b)\text{ is an inversion of }\pi\}.   \]
Hence, using the bijection $\pi\mapsto \pi B$, weak order is isomorphic to the poset of regions of $\cB_n$.

Given any permutation $\pi$, we can consider the region of the braid arrangement $\pi B$. The lower walls $(a,b)$ of $\pi$ correspond to the hyperplanes $H_{ab}$ in $\cS(B,\pi B)$ which are incident to $\pi B$. (Hence the term \emph{wall}.)

\subsection{Shards}
\begin{figure}
    \centering
    \begin{minipage}{.5\textwidth}
    \centering
    \begin{tikzpicture}[scale=1.7]
        \draw[thick] (0:-1.7) -- (0:1.7);
        \draw[thick] (60:-1.7) -- (60:1.7);
        \draw[thick] (120:-1.7) -- (120:1.7);
        
        \node (B) at (270:1.4) {$x_1<x_2<x_3$};
        \node (L1) at (210:1.4) {$x_2<x_1<x_3$};
        \node (L2) at (150:1.4) {$x_2<x_3<x_1$};
        \node (T) at (90:1.4) {$x_3<x_2<x_1$};
        \node (R2) at (30:1.4) {$x_3<x_1<x_2$};
        \node (R1) at (-30:1.4) {$x_1<x_3<x_2$};

        \node[below] at (60:-1.7) {$H_{12}$};
        \node[below] at ($(120:-1.7)+(.1,0)$) {$H_{23}$};
        \coordinate (CBSW) at (current bounding box.south west);
        \coordinate (CBNE) at (current bounding box.north east);
        \node[right] at (-7:1.4) {$H_{13}$};
        \draw[thick] (-2,-2) rectangle (2,2);
    \end{tikzpicture}
    \end{minipage}%
    \begin{minipage}{.5\textwidth}
    \centering
    \begin{tikzpicture}[scale=1.7] 
        \draw[thick] (0:-1.7) -- (0:-.1);
        \draw[thick] (0:.1) -- (0:1.7);
        \draw[thick] (60:-1.7) -- (60:1.7);
        \draw[thick] (120:-1.7) -- (120:1.7);
    


        \begin{scope}[shift={(270:1.2)}, scale=0.3, every node/.style={scale=0.75}]
        \node (A1) at (-1,0) {$1$};
    \node (A2) at (0,0) {$2$};
    \node (A3) at (1,0) {$3$};

        \end{scope}
        \begin{scope}[shift={(210:1.2)}, scale=0.3, every node/.style={scale=0.75}]
        \node (A1) at (-1,0) {$1$};
    \node (A2) at (0,0) {$2$};
    \node (A3) at (1,0) {$3$};

    \draw (A1.east) to[out=0,in=180] (A2.west);
        \end{scope}
        \begin{scope}[shift={(90:1.2)}, scale=0.3, every node/.style={scale=0.75}]
        \node (A1) at (-1,0) {$1$};
    \node (A2) at (0,0) {$2$};
    \node (A3) at (1,0) {$3$};

    \draw (A1.east) to[out=0,in=180] (A2.west);
    \draw (A2.east) to[out=0,in=180] (A3.west);        
        \end{scope}
        \begin{scope}[shift={(150:1.2)}, scale=0.3, every node/.style={scale=0.75}]
        \node (A1) at (-1,0) {$1$};
    \node (A2) at (0,0) {$2$};
    \node (A3) at (1,0) {$3$};

\draw (A1.east) to[out=0,in=180] (0,.5) to[out=0,in=180] (A3.west);      
        \end{scope}
        \begin{scope}[shift={(30:1.2)}, scale=0.3, every node/.style={scale=0.75}]
        \node (A1) at (-1,0) {$1$};
    \node (A2) at (0,0) {$2$};
    \node (A3) at (1,0) {$3$};

    \draw (A1.east) to[out=0,in=180] (0,-.5) to[out=0,in=180] (A3.west);        
        \end{scope}
        \begin{scope}[shift={(-30:1.2)}, scale=0.3, every node/.style={scale=0.75}]
        \node (A1) at (-1,0) {$1$};
    \node (A2) at (0,0) {$2$};
    \node (A3) at (1,0) {$3$};

    \draw (A2.east) to[out=0,in=180] (A3.west);
        \end{scope}

        \node[below] at (60:-1.7) {$\Sigma(1,2,\varnothing,\varnothing)$};
        \node[below] at ($(120:-1.7)+(.1,0)$) {$\Sigma(2,3,\varnothing,\varnothing)$};
        \coordinate (CBSW) at (current bounding box.south west);
        \coordinate (CBNE) at (current bounding box.north east);
        \node[right] at (-15:.5) {$\Sigma(1,3,\varnothing,\{2\})$};
        \node[left] at (15:-.5) {$\Sigma(1,3,\{2\},\varnothing)$};
        \draw[thick] (-2,-2) rectangle (2,2);
    \end{tikzpicture}
    \end{minipage}
    \begin{minipage}[t]{.5\textwidth}
    \vspace{1em}
    \captionsetup{width=.9\linewidth}
    \caption{The intersection of $\cB_3$ with a two-dimensional subspace of $\RR^3$. 
    Two points are in the same region of $\cB_3$ if their coordinates are in the same order.
    }
    \label{fig:braid}
    \end{minipage}%
    \begin{minipage}[t]{.5\textwidth}
    \vspace{1em}
    \captionsetup{width=.9\linewidth}
    \caption{The four shards of $\cB_3$ are the two hyperplanes $H_{12},H_{23}$ and the two halves of $H_{13}$. The regions are labeled by the lower arc diagram of their associated permutation.}
    \label{fig:shardsSn}
    \end{minipage}
\end{figure}

In this subsection, we will relate arcs to the geometry of the braid arrangement. Given an arc for $S_n$, we will associate a polyhedral cone in the braid arrangement $\cB_n$. 
To do so, define the half-spaces
\[H_{ab}^+ \coloneqq \{ (x_1,\ldots,x_n)\mid x_a< x_b\}  \qquad
 H_{ab}^- \coloneqq \{ (x_1,\ldots,x_n) \mid x_a> x_b\}. \]
Consider the arc $\aa = (a,c,L,R)$. For each $b\in L\cup R$, we define the half-spaces 
\[ H_{ab}^\aa \coloneqq \begin{cases}
    H_{ab}^+ &\text{if $c \in R$} \\
    H_{ab}^- &\text{if $c \in L$}
\end{cases}, \qquad H_{bc}^\aa \coloneqq \begin{cases}
    H_{bc}^+ &\text{if $c \in L$} \\
    H_{bc}^- &\text{if $c \in R$}
\end{cases}. \]
Now, we define the polyhedral cone
\[ \Sigma_\aa \coloneqq H_{ac}\cap \bigcap_{b\in L\cup R} H_{ab}^\aa = H_{ac}\cap\bigcap_{b\in L\cup R} H_{bc}^\aa. \]

The cones we construct this way are called \emph{shards}, and have a lattice-theoretic significance. First we will characterize them geometrically.
If $\cA$ is a hyperplane arrangement with a fixed choice of base region $B$, then we write $H^+$ to be the unique half-space bounded by $H$ and containing $B$. We write $H^-$ for the other half-space, so that $\RR^n = H^+ \sqcup H \sqcup H^-$.
\begin{definition}
    Let $\cA$ be a hyperplane arrangement with base region $B$. A hyperplane $H\in\cA$ is \newword{basic} if there is a line segment with one endpoint in $H$ and one endpoint in $B$, which is disjoint from all other hyperplanes. For distinct hyperplanes $H_1,H_2\in \cA$, the rank 2 arrangement \newword{spanned} by $H_1,H_2$ is $\cA'=\{H \in \cA \mid H\supseteq H_1\cap H_2 \}$.  We say that $H_1$ \newword{cuts} $H_2$ if $H_1$ is basic in $\cA'$ and $H_2$ is not basic in $\cA'$. If $H\in \cA$, then the connected components of $H\setminus \bigcup_{H'\text{ cuts } H} H'$ are called \newword{shards} of $H$. We write $\Sha(\cA)$ for the set of shards of hyperplanes in $\cA$.
\end{definition}

In the braid arrangement $\cB_n$, the rank two arrangement $\{H_{ab},H_{bc},H_{ac}\}$ is spanned by any two of its elements and has basic hyperplanes $H_{ab}$ and $H_{bc}$. If $a,b,c,d$ are pairwise distinct, then $H_{ab},H_{cd}$ span the rank two arrangement $\{H_{ab},H_{cd}\}$ and both of its hyperplanes are basic. Hence the hyperplanes which cut $H_{ac}$ are those of the form $H_{ab}$ or $H_{bc}$, for $a<b<c$. We conclude that the shards of $H_{ac}$ are the connected components of 
\[ H_{ac} \setminus \bigcup_{a<b<c} H_{ab} = H_{ac} \setminus \bigcup_{a<b<c} H_{bc}. \]
It straightforward to check that each $\Sigma_\aa$ is such a connected component. More is true.

\begin{proposition}
    The map $\aa \mapsto \Sigma_\aa$ is a bijection from arcs to $\Sha(\cB_n)$.
\end{proposition}
If $\aa = (a,b,L,R)$ then we also denote $\Sigma_\aa$ by $\Sigma(a,b,L,R)$.

\begin{example}
In $\cB_3$, there are four total shards: $H_{12}$ and $H_{23}$ are themselves shards, and $H_{13}$ is (the closure of) the union of two shards. In \Cref{fig:shardsSn}, the two shards in $H_{13}$ are the left and right halves of $H_{13}$. The left half is $\Sigma(1,3,\{2\},\varnothing)$ and the right half is $\Sigma(1,3,\varnothing,\{2\})$.
\end{example}

Now let's connect the combinatorics to the geometry. If $\pi$ is a permutation, and $\aa= (a,b,L,R)$ is a lower arc of $\pi$, then in particular $(a,b)$ is a lower wall of $\pi$. Hence the hyperplane $H_{ab}$ is incident to the region $\pi B$ in the braid arrangement. The hyperplane is a union of shards, and there is a unique shard $\Sigma$ so that $\Sigma$ is incident to $\pi B$ (in the sense that $\Sigma\cap \overline{\pi B}$ is full-dimensional in $\Sigma$). That unique shard $\Sigma$ is precisely $\Sigma_\aa$. If a shard $\Sigma$ is incident to $\pi B$ and contained in a lower wall of $\pi B$, then we say that $\Sigma$ is a \newword{lower shard} of $\pi B$.

\begin{example}
Consider the \JI $\pi=231$. Then the unique lower wall of $\pi$ is $(1,3)$, so the unique lower shard of $\pi B$ should be contained in $H_{13}$. Examining \Cref{fig:shardsSn}, we see that the unique lower shard of $\pi B$ is the left half of $H_{13}$, which has shard arc \begin{tikzpicture}[scale=.5]
\node[inner sep=0] (A1) at (0,0) {$1$};
\node[inner sep=0] at (1,0) {$2$};
\node[inner sep=0] (A3) at (2,0) {$3$};
\draw (0.25,0) to[out=0,in=180] (1,.75) to[out=0,in=180] (1.75,0);
\end{tikzpicture}. As expected, this is the shard arc associated to $231$.
\end{example}

\subsection{Lattice congruences on $S_n$}
\begin{definition}
Let $L$ be a complete lattice. A \newword{complete lattice congruence} is an equivalence relation $\equiv$ on $L$ so that if $x_i\equiv y_i$ for all $i\in I$, then $\bigvee_{i\in I} x_i = \bigvee_{i\in I} y_i$ and $\bigwedge_{i\in I} x_i = \bigwedge_{i\in I} y_i$. The \newword{quotient lattice} $L/{\equiv}$ is the set of $\equiv$-equivalence classes with their induced ordering. 

If $L'$ is also a complete lattice, then a \newword{complete lattice homomorphism} is a function $\eta:L\to L'$ so that $\eta(\bigwedge_{i\in I} x_i) = \bigwedge_{i\in I} \eta(x_i)$ and $\eta(\bigvee_{i\in I} x_i) = \bigvee_{i\in I} \eta(x_i)$.
\end{definition}

Given any complete lattice congruence $\equiv$ on a complete lattice $L$, there is a surjective complete lattice homomorphism $\eta: L \to L/{\equiv}$.

\begin{proposition}\label{prop:congruence}
Let $L$ be a complete lattice and $\equiv$ an equivalence relation on $L$. Then the following are equivalent:
\begin{itemize}
	\item[(a)] $\equiv$ is a complete lattice congruence on $L$;
	\item[(b)] There is a lattice $L'$ and a complete lattice homomorphism $\eta: L \to L'$ so that $x\equiv y$ if and only if $\eta(x)=\eta(y)$;
	\item[(c)] The following conditions all hold:
	\begin{itemize}
		\item Each equivalence class of $\equiv$ is an interval in $L$, and
		\item If $\pi^\downarrow(x)$ denotes the unique minimal element of $[x]_\equiv$, then $\pi^\downarrow : L \to L$ is order-preserving, and
		\item If $\pi^\uparrow(x)$ denotes the unique maximal element of $[x]_\equiv$, then $\pi^\uparrow : L \to L$ is order-preserving.
	\end{itemize}
\end{itemize}
\end{proposition}

If $x$ is an element of the quotient lattice $L/{\equiv}$, then we will often abuse notation and write $\pi^{\downarrow}_\equiv (x)$ for the value of $\pi^{\downarrow}_\equiv$ on any representative of $x$.

\begin{corollary}\label{prop:sublatticecongruence}
    Let $L$ be a complete lattice and $L'\subseteq L$ a complete sublattice. Then any complete lattice congruence on $L$ restricts to a complete lattice congruence on $L'$.
\end{corollary}

Let $\aa = (a,b,L,R)$ and $\aa'=(a',b',L',R')$ be arcs for $S_n$. Then we say that $\aa$ is a \newword{subarc} of $\aa'$ if $a \leq a' < b' \leq b$ and $L\subseteq L'$ and $R\subseteq R'$. Reading uses arcs to parametrize the lattice congruences of $S_n$. A congruence $\equiv$ on $S_n$ is said to \newword{contract} the arc $\aa=(a,b,L,R)$ if there is some $\pi$ so that $\aa\in \LowerArcs(\pi)$ and $(a,b)\cdot \pi \equiv \pi$. If $\aa$ is contracted by $\equiv$, then any $\pi$ having $\aa$ as a lower arc satisfies $(a,b)\cdot \pi \equiv \pi$. Write $\Arcs(\equiv)$ for the set of arcs which are not contracted by $\equiv$. We say a set of arcs $X$ is closed under taking subarcs if, whenever $\aa$ is a subarc of $\bb$ and $\bb\in X$, then $\aa$ in $X$.

\begin{proposition}
The map ${\equiv} \mapsto \Arcs(\equiv)$ is a bijection from complete lattice congruences on $S_n$ to sets of arcs which are closed under taking subarcs. 
\end{proposition}

\subsection{The affine symmetric group}
\begin{figure}
\centering
\begin{tikzpicture}
	\draw (-3.4,-3.4) rectangle (3.4,.7);	

		\draw[fill] (-90:2.3) circle (1.2pt) -- ({.5*atan2(0,0+1)+.5*atan2(1,1+1)-45}:2.3) circle (1.2pt)
		-- ({.5*atan2(1,1+1)+.5*atan2(2,2+1)-45}:2.3) circle (1.2pt)
		-- ({.5*atan2(2,2+1)+.5*atan2(3,3+1)-45}:2.3) circle (1.2pt)
		-- ({.5*atan2(3,3+1)+.5*atan2(4,4+1)-45}:2.3) circle (.8pt)
		-- ({.5*atan2(4,4+1)+.5*atan2(5,5+1)-45}:2.3) circle (.6pt);

        \draw (-90:2.7) node {$[1,2]$};
        \draw ({.5*atan2(0,0+1)+.5*atan2(1,1+1)-45}:2.8) node[scale=.8] {$[2,1]$};
        \draw ({.5*atan2(1,1+1)+.5*atan2(2,2+1)-48}:2.7) node[scale=.6] {$[-1,4]$};
        \draw ({.5*atan2(2,2+1)+.5*atan2(3,3+1)-45}:2.7) node[scale=.5] {$[4,-1]$};
        
		\foreach \x in {1,2,...,3} {
			\fill ({atan2(4.5,4.5+1)-45+1.4668*\x}:2.3) circle (.6pt);
		};

		\draw[fill] (-90:2.3) circle (1.2pt) -- ({.5*atan2(0+1,0)+.5*atan2(1+1,1)-45}:-2.3) circle (1.2pt)
		-- ({.5*atan2(1+1,1)+.5*atan2(2+1,2)-45}:-2.3) circle (1.2pt)
		-- ({.5*atan2(2+1,2)+.5*atan2(3+1,3)-45}:-2.3) circle (1.2pt)
		-- ({.5*atan2(3+1,3)+.5*atan2(4+1,4)-45}:-2.3) circle (.8pt)
		-- ({.5*atan2(4+1,4)+.5*atan2(5+1,5)-45}:-2.3) circle (.6pt);

        \draw ({.5*atan2(0+1,0)+.5*atan2(1+1,1)-45}:-2.8) node[scale=.8] {$[0,3]$};
        \draw ({.5*atan2(1+1,1)+.5*atan2(2+1,2)-42}:-2.65) node[scale=.6] {$[3,0]$};
        \draw ({.5*atan2(2+1,2)+.5*atan2(3+1,3)-45}:-2.7) node[scale=.5] {$[-2,5]$};

		\foreach \x in {1,2,...,3} {
			\fill ({atan2(4.5+1,4.5)-45-1.4668*\x}:-2.3) circle (.6pt);
		};
\end{tikzpicture}
\caption{The Hasse diagram of weak order on $\tS_2$. Each affine permutation is labeled by its window notation.}
\label{fig:S2tildeweak}
\end{figure}
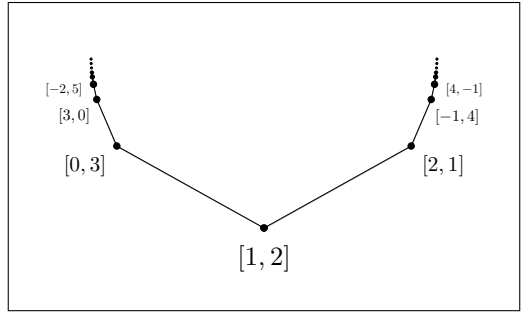
We now briefly introduce the affine symmetric group, which is a Coxeter group that will feature in later sections.
\begin{definition}\label{def:Stilde}
    The \newword{affine symmetric group} $\tS_n$ is the group of bijections $\tpi:\ZZ\to\ZZ$ satisfying:
    \begin{itemize}
        \item[(a)] $\tpi(i+n) = \tpi(i) + n$ for all $i\in \ZZ$, and
        \item[(b)] $\displaystyle\sum_{i=1}^n\tpi(i) = \sum_{i=1}^ni$.
    \end{itemize}
\end{definition}
Elements of $\tS_n$ are \newword{affine permutations}. The one-line notation of an affine permutation is defined similarly to a usual permutation, so that for instance the one-line notation of the identity is:
\[ \ldots,-1,0,1,2,3,4,5,\ldots. \]
Condition (b) in \Cref{def:Stilde} lets us recover an affine permutation from its one-line notation. We abbreviate affine permutations via \newword{window notation}: given a sequence of $n$ integers $x_1,\ldots,x_n$ which have distinct residue classes mod $n$, we write $[x_1,x_2,\ldots, x_n]$ for the unique affine permutation whose one-line notation contains $x_1,\ldots,x_n$ as a consecutive subsequence. For instance, in $\tS_3$, the windows $[1,2,3]$ and $[0,1,2]$ both represent the identity permutation, whereas $[1,0,2]$ represents the permutation
\[ \ldots,-3,-1,1,0,2,4,3,5,\ldots. \]
The window notations for the elements of $\tS_2$ are shown in \Cref{fig:S2tildeweak}.
We can partially order $\tS_2$ by the \newword{weak order}, which puts $\tpi_1 \leq \tpi_2$ if $\tpi_1^{-1}(b) < \tpi_1^{-1}(a)$ implies $\tpi_2^{-1}(b) < \tpi_2^{-1}(b)$ for all pairs $(a,b)$ with $a<b$.

\subsection{Extended weak order}
In this section we describe the extended weak order on a Coxeter group, which is a poset introduced by Matthew Dyer. 

\begin{definition}
Let $V$ be a real vector space and $\Psi \subseteq V$ be a set of vectors. A subset $X\subseteq \Psi$ is \newword{closed} if, for any vectors $\alpha,\beta,\gamma\in \Psi$ so that $\gamma \in \RR_{\geq 0}\alpha + \RR_{\geq 0}\beta$, if $\alpha$ and $\beta$ are both in $X$ then so is $\gamma$. A subset $X\subseteq \Psi$ is \newword{coclosed} if $\Psi\setminus X$ is closed. A subset $X\subseteq \Psi$ is \newword{biclosed} if $X$ is closed and coclosed. The \newword{Dyer poset} of $\Psi$, denoted $\Dyer(\Psi)$, is the poset of biclosed sets under containment order.
\end{definition}

Each Coxeter group $W$ has an associated \emph{positive root system}, which is a set of vectors $\Phi^+_W$. We abbreviate $\Dyer(\Phi^+_W)$ to $\Dyer(W)$. In this case the Dyer poset is also called the \newword{extended weak order} of $W$. Note that the extended weak order of $W$ is not a partial order on $W$, but is rather a partial order on the biclosed subsets of $\Phi^+_W$.

\begin{figure}
    \centering
    \begin{minipage}{.5\textwidth}
    \centering
    \begin{tikzpicture}
					\node (B) at (270:1.4) {$\varnothing$};
					\node (L1) at (210:1.4) {};
					\node (L2) at (150:1.4) {};
					\node (T) at (90:1.4) {};
					\node (R2) at (30:1.4) {};
					\node (R1) at (-30:1.4) {};
					\draw (B) -- (L1) -- (L2) -- (T) -- (R2) -- (R1) -- (B);
					
					\begin{scope}[shift=(T),scale=.3]
						\fill[fill=white] (0,0) circle (1);
						\draw[thick,blue,-stealth] (0,0) -- (0+30:1);
						\draw[thick,blue,-stealth] (0,0) -- (120+30:1);
						\draw[thick,blue,-stealth] (0,0) -- (60+30:1);
					\end{scope}
					
					\begin{scope}[shift=(L1),scale=.3]
						\fill[fill=white] (0,0) circle (1);
						\draw[thick,blue,-stealth] (0,0) -- (120+30:1);
					\end{scope}
					
					\begin{scope}[shift=(L2),scale=.3]
						\fill[fill=white] (0,0) circle (1);
						\draw[thick,blue,-stealth] (0,0) -- (120+30:1);
						\draw[thick,blue,-stealth] (0,0) -- (60+30:1);
					\end{scope}
					
					\begin{scope}[shift=(R1),scale=.3]
						\fill[fill=white] (0,0) circle (1);
						\draw[thick,blue,-stealth] (0,0) -- (0+30:1);
					\end{scope}
					
					\begin{scope}[shift=(R2),scale=.3]
						\fill[fill=white] (0,0) circle (1);
						\draw[thick,blue,-stealth] (0,0) -- (0+30:1);
						\draw[thick,blue,-stealth] (0,0) -- (60+30:1);
					\end{scope}
				\end{tikzpicture}
    \end{minipage}%
    \begin{minipage}{.5\textwidth}
    \centering
    \begin{tikzpicture}[yscale=1]
				\node (B) at (0,0) {$\varnothing$};
				\draw[fill,] (B) 
				-- ++(1.707,.307) node[right] (R1) {} circle (1.2pt)
				(-1.707,.307) node[left] (L1) {} circle (1.2pt) -- (B);
				\draw[fill,] (-1.707,.307) -- ++(0,.5) node[left] (L2) {} circle (1.2pt)
				-- ++ (0, .3) circle (1.2pt) -- ++(0, .2) circle (1.2pt)
				++ (0, .1) circle (.3pt) ++ (0, .1) circle (.3pt)
				++ (0, .1) circle (.3pt);
				\draw[fill,] (1.707,.307) -- ++(0,.5) node[right] (R2) {} circle (1.2pt)
				-- ++ (0, .3) circle (1.2pt) -- ++(0, .2) circle (1.2pt)
				++ (0, .1) circle (.3pt) ++ (0, .1) circle (.3pt)
				++ (0, .1) circle (.3pt);
				
				\draw[fill] (-1.707, .307+1.4) node[left] (ML) {} circle (1.2pt)
				(1.707, .307+1.4) node[right] (MR) {} circle (1.2pt);
				
				\begin{scope}[shift={(0,.307+1.4)}]
					\draw[fill] (-1.707,.1) circle (.3pt) ++ (0,.1) circle (.3pt) ++ (0,.1) circle (.3pt) ++ (0, .1) circle (1.2pt) -- ++ (0,.2) circle (1.2pt) 
					-- ++(0,.3) node[left] (L3) {} circle (1.2pt)
					-- ++(0,.5) node[left] (L4) {} circle (1.2pt)
					-- ++(1.707,.307) node[above] (T) {} circle (1.2pt);
					
					\draw[fill] (1.707,.1) circle (.3pt) ++ (0,.1) circle (.3pt) ++ (0,.1) circle (.3pt) ++ (0, .1) circle (1.2pt) -- ++ (0,.2) circle (1.2pt) 
					-- ++(0,.3) node[right] (R3) {} circle (1.2pt)
					-- ++(0,.5) node[right] (R4) {} circle (1.2pt)
					-- ++(-1.707,.307);
				\end{scope}
			
				\begin{scope}[shift=(T),scale=.2]
					\clip (-2.5,0) rectangle (2.5,4);
					\foreach \x in {0,1,...,20} {
						\draw[blue,-stealth] (0,0) -- ({atan2(\x,\x+1)+45}:{sqrt(\x^2+(\x+1)^2)});
						\draw[blue,-stealth] (0,0) -- ({atan2(\x+1,\x)+45}:{sqrt(\x^2+(\x+1)^2)});
					};
					\fill[blue] (0,0) -- ({atan2(20+1,20)+45}:{sqrt(20^2+(20+1)^2)}) -- (90:4);
					\fill[blue] (0,0) -- ({atan2(20,20+1)+45}:{sqrt(20^2+(20+1)^2)}) -- (90:4);
				\end{scope}
			
				\begin{scope}[shift=(R1),scale=.2]
					\clip (-2.5,0) rectangle (2.5,4);
					\foreach \x in {0} {
						\draw[blue,-stealth] (0,0) -- ({atan2(\x,\x+1)+45}:{sqrt(\x^2+(\x+1)^2)});
					};
				\end{scope}
			
				\begin{scope}[shift=(R2),scale=.2]
					\clip (-2.5,0) rectangle (2.5,4);
					\foreach \x in {0,1} {
						\draw[blue,-stealth] (0,0) -- ({atan2(\x,\x+1)+45}:{sqrt(\x^2+(\x+1)^2)});
					};
				\end{scope}
			
				\begin{scope}[shift=(ML),scale=.2]
					\clip (-2.5,0) rectangle (2.5,4);
					\foreach \x in {0,1,...,20} {
						\draw[blue,-stealth] (0,0) -- ({atan2(\x+1,\x)+45}:{sqrt(\x^2+(\x+1)^2)});
					};
					\fill[blue] (0,0) -- ({atan2(20+1,20)+45}:{sqrt(20^2+(20+1)^2)}) -- (90:4);
				\end{scope}
			
				\begin{scope}[shift=(MR),scale=.2]
					\clip (-2.5,0) rectangle (2.5,4);
					\foreach \x in {0,1,...,20} {
						\draw[blue,-stealth] (0,0) -- ({atan2(\x,\x+1)+45}:{sqrt(\x^2+(\x+1)^2)});
					};
					\fill[blue] (0,0) -- ({atan2(20,20+1)+45}:{sqrt(20^2+(20+1)^2)}) -- (90:4);
				\end{scope}
			
				\begin{scope}[shift=(L1),scale=.2]
					\clip (-2.5,0) rectangle (2.5,4);
					\foreach \x in {0} {
						\draw[blue,-stealth] (0,0) -- ({atan2(\x+1,\x)+45}:{sqrt(\x^2+(\x+1)^2)});
					};
				\end{scope}
				
				\begin{scope}[shift=(L2),scale=.2]
					\clip (-2.5,0) rectangle (2.5,4);
					\foreach \x in {0,1} {
						\draw[blue,-stealth] (0,0) -- ({atan2(\x+1,\x)+45}:{sqrt(\x^2+(\x+1)^2)});
					};
				\end{scope}
			
				\begin{scope}[shift=(R4),scale=.2]
					\clip (-2.5,0) rectangle (2.5,4);
					\foreach \x in {0,1,...,20} {
						\draw[blue,-stealth] (0,0) -- ({atan2(\x,\x+1)+45}:{sqrt(\x^2+(\x+1)^2)});
						\draw[blue,-stealth] (0,0) -- ({atan2(\x+2,\x+1)+45}:{sqrt((\x+1)^2+(\x+2)^2)});
					};
					\fill[blue] (0,0) -- ({atan2(20+1,20)+45}:{sqrt(20^2+(20+1)^2)}) -- (90:4);
					\fill[blue] (0,0) -- ({atan2(20,20+1)+45}:{sqrt(20^2+(20+1)^2)}) -- (90:4);
				\end{scope}
			
				\begin{scope}[shift=(L4),scale=.2]
					\clip (-2.5,0) rectangle (2.5,4);
					\foreach \x in {0,1,...,20} {
						\draw[blue,-stealth] (0,0) -- ({atan2(\x+1,\x+2)+45}:{sqrt((\x+1)^2+(\x+2)^2)});
						\draw[blue,-stealth] (0,0) -- ({atan2(\x+1,\x+0)+45}:{sqrt((\x+0)^2+(\x+1)^2)});
					};
					\fill[blue] (0,0) -- ({atan2(20+1,20)+45}:{sqrt(20^2+(20+1)^2)}) -- (90:4);
					\fill[blue] (0,0) -- ({atan2(20,20+1)+45}:{sqrt(20^2+(20+1)^2)}) -- (90:4);
				\end{scope}
			\end{tikzpicture}
    \end{minipage}
    \begin{minipage}[t]{.5\textwidth}
    \vspace{1em}
    \captionsetup{width=.9\linewidth}
    \caption{The poset $\Dyer(S_3)$.}
    \label{fig:DyerS3}
    \end{minipage}%
    \begin{minipage}[t]{.5\textwidth}
    \vspace{1em}
    \captionsetup{width=.9\linewidth}
    \caption{The poset $\Dyer(\tS_2)$. 
    }
    \label{fig:DyertS2}
    \end{minipage}
\end{figure}

\begin{example}
Let $\be_1,\ldots,\be_n$ be the standard basis of $\RR^n$. 
Define the vectors $\alpha_{a,b} \coloneqq \be_b-\be_a$. Then the positive root system of $S_n$ lives in $\RR^n$ as the subset
\[ \Phi^+_{S_n} \coloneqq \{ \alpha_{a,b} \mid 1 \leq a < b \leq n \}. \]
The biclosed sets in $\Phi^+_{S_3}$ are shown in \Cref{fig:DyerS3}. The biclosed sets in $\Phi^+_{S_n}$ are easy to describe. The following has been proven in many forms: early results include \cite{Cellini} and \cite{Ito}.
\begin{proposition}
A subset $X$ of $\Phi^+_{S_n}$ is biclosed if and only if 
\[X = \{ \alpha_{a,b} \mid (a,b) \in N(\pi) \} \] 
for some permutation $\pi\in S_n$. In particular, the posets $\Dyer(S_n)$ and $\WO(S_n)$ are isomorphic.
\end{proposition}
\end{example}

\begin{example}\label{tSnroots}
Let $\widetilde\RR^{n}$ denote the $(n+1)$-dimensional vector space with basis vectors $\tbe_1,\ldots,\tbe_n,\delta$. Define $\tbe_{i+kn}\coloneqq \tbe_i + k\delta$. In this way, $\tbe_i$ denotes some vector in $\widetilde\RR^n$ for each $i\in\ZZ$.  Define the vectors $\talpha_{a,b}\coloneqq \tbe_b - \tbe_a$.
Then the positive root system of the affine symmetric group $\tS_n$ lives in $\widetilde\RR^n$ as the subset
\[ \{ \talpha_{a,b} \mid a<b,~ a\not\equiv b \mod n \}. \]
The biclosed subsets of $\Phi^+_{\tS_2}$ are shown in \Cref{fig:DyertS2}. The biclosed sets in $\Phi^+_{\tS_n}$ are no longer easy to describe; we shall recall their characterization in \Cref{thm:TITOtoDyer}. However, the \emph{finite} biclosed sets admit a simple description. 
\begin{proposition}
A finite subset $X$ of $\Phi^+_{\tS_n}$ is biclosed if and only if 
\[ X = \{ \talpha_{a,b} \mid a < b \text{ and } \tpi^{-1}(b) < \tpi^{-1}(a) \} \]
for some $\tpi \in \tS_n$.
In particular, $\WO(\tS_n)$ is isomorphic to the subposet of $\Dyer(\tS_n)$ consisting of finite biclosed sets.
\end{proposition}
\end{example}

\section{The infinite symmetric group $S_\infty$}\label{sec:Sinfty}
In this section, we extend the theory of arc diagrams and canonical join representations to the Coxeter group $S_\infty$. We will use the results of this section to study translation-invariant total orders later, but studying $S_\infty$ itself is also a good introduction to some poset-theoretic difficulties that arise in the study of extended weak order. 
\begin{definition}
    The \newword{infinite symmetric group} $S_\infty$ is the group of bijections $\pi : \ZZ \to \ZZ$ such that $\{i\in \ZZ \mid \pi(i)\neq i\}$ is a finite set.
\end{definition}
We will sometimes depict elements of $S_\infty$ via a finite one-line notation such as $23145$. This is to be interpreted as the permutation which fixes all integers outside the set $\{1,\ldots,5\}$, and acts on $\{1,\ldots,5\}$ as the permutation $23145$.
Just as with the finite symmetric group $S_n$, an \newword{inversion} of a permutation $\pi\in S_\infty$ is a pair $(a,b)$ with $a<b$ so that $\pi^{-1}(b) < \pi^{-1}(a)$. The \newword{weak order} on $S_\infty$ puts $\pi_1 \leq \pi_2$ if there is a containment of inversion sets $N(\pi_1)\subseteq N(\pi_2)$. Note that $N(\pi)$ is always a finite set.

\begin{definition}
    A \newword{total order} on $\ZZ$ is a relation $(\prec)$ which is:
    \begin{itemize}
        \item Transitive: for all $a,b,c\in\ZZ$, if $a\prec b$ and $b\prec c$, then $a\prec c$, and
        \item Trichotomous: for all $a,b\in \ZZ$, exactly one of $a\prec b$, $a=b$, or $b \prec a$ occur.
    \end{itemize}
    We write $\Tot$ for the set of total orders on $\ZZ$. If $(\prec)$ is such a total order, then an \newword{inversion} of $\prec$ is a pair $(a,b)$ with $a<b$ so that $b\prec a$. The set of inversions of $(\prec)$ is denoted $N(\prec)$. The \newword{weak order} on $\Tot$ is the partial order on $\Tot$ which puts $(\prec_1) \leq (\prec_2)$ if $N(\prec_1)\subseteq N(\prec_2)$.
\end{definition}

In this definition (and everywhere in this paper), if $a,b\in \ZZ$ then we write $a<b$ for the usual ordering of the integers, and reserve the symbol $\prec$ for possibly-exotic total orders. We will write $[a,b]_{\prec}$ for the closed interval $\{x\in\ZZ\mid a\preceq x \preceq b\}$. 
Given a permutation $\pi$, we can construct a total order $(\prec_\pi)$ so that $a\prec_\pi b$ if and only if $\pi^{-1}(a) < \pi^{-1}(b)$. Equivalently, $a\prec_\pi b$ if and only if $a$ appears to the left of $b$ in the one-line notation of $\pi$. Then by comparing definitions, we have that $N(\prec_\pi) = N(\pi)$. Hence the map $\pi \mapsto (\prec_\pi)$ is an order-preserving embedding $\WO(S_\infty) \to \WO(\Tot)$, which realizes $\WO(S_\infty)$ as the order ideal of $\WO(\Tot)$ consisting of total orders with finite inversion sets. 

\begin{aside}
\begin{remark}
    It was shown in \cite{Barkley2022} that $\WO(\Tot)$ is isomorphic to $\Dyer(S_\infty)$, the extended weak order of $S_\infty$. More precisely, a subset $X$ of $\Phi^+_{S_\infty}$ is biclosed if and only if $X= N(\prec)$ for a total order $\prec$.
\end{remark}
\end{aside}

The following is an analog of \Cref{prop:latticeSn}, shown in \cite[Theorem 6.4]{Barkley2022}. Let $T \coloneqq \{(a,b)\mid a< b\}$. The \newword{closure} of a subset $N\subseteq T$ is the unique minimal set $\overline{N}\supseteq N$ so that whenever $(a,b)$ and $(b,c)$ are in $\overline{N}$, then so is $(a,c)$.

\begin{proposition}\label{prop:latticeSinfty}
    The poset $\WO(\Tot)$ is a complete lattice. The join of a family $(\prec_i)_{i\in I}$ is the unique total order $(\prec)$ with 
    \[ N(\prec) = \overline{\bigcup_{i\in I} N(\prec_i)}.  \]
\end{proposition}

In contrast, the poset $\WO(S_\infty)$ is \emph{not} a complete lattice\footnote{It is true that $\WO(S_\infty)$ is a non-complete lattice. For a finite-rank Coxeter group $W$, weak order $\WO(W)$ is a lattice if and only if it is a complete lattice if and only if $W$ is finite, but $S_\infty$ has infinite rank.}.

\begin{example}\label{ex:joinSinfty}
Consider the sequence of permutations in $S_\infty$
\begin{align*}
    \pi_1 &= 1\mathbf023456 &  N(\pi_1)&= \{(0,1)\}  \\
    \pi_2 &= 12\mathbf03456 & N(\pi_2)&= \{(0,1),(0,2)\}\\
    \pi_3 &= 123\mathbf0456 &N(\pi_3)&=\{(0,1),(0,2),(0,3)\}\\
    \pi_4 &= 1234\mathbf056 &N(\pi_4)&=\{(0,1),(0,2),(0,3),(0,4)\}\\
    &\phantom{.}\vdots &\vdots
    \end{align*}
Each permutation in the sequence moves the integer $0$ one step further to the right. 
Because $N(\pi_i)\subsetneq N(\pi_{i+1})$, this sequence of permutations is strictly increasing in $\WO(S_\infty)$. There is no join for this sequence in the weak order on $S_\infty$. Such a join would have to have every pair $(0,b)$ as an inversion, but any element of $S_\infty$ has a finite inversion set. Intuitively, the join should put $0$ to the right of every other integer, and no permutation can do that.
However, since $\WO(\Tot)$ is a complete lattice, it has a least upper bound for the sequence of total orders $(\prec_{\pi_1}),(\prec_{\pi_2}),(\prec_{\pi_3}),\ldots$. 

This least upper bound should be the minimal total order $(\prec)$ so that $N(\prec)$ contains 
\[\bigcup_{i\in \NN} N(\prec_{\pi_i}) = \{(0,1),(0,2),(0,3),\ldots\}.\] 
In fact there is a total order $(\prec)$ so that $N(\prec)$ is exactly $\{(0,1),(0,2),(0,3),\ldots\}$, which must therefore be the join. We conclude that $\bigvee_{i\in\NN} (\prec_{\pi_i})$ is the total order
\[ \cdots\prec -3\prec -2\prec -1\prec 1\prec 2\prec 3\prec\cdots\prec 0. \]
\end{example}

\subsection{Walls of total orders}

\begin{definition}
Let $(\prec)$ be a total order of $\ZZ$. A pair $(a,b)$ with $a<b$ is a \newword{lower wall} of $(\prec)$ if $b\precdot a$ is a cover relation (i.e., there is no integer $c$ with $b\prec c \prec a$). An arc $(a,b,L,R)$ is a \newword{lower arc} of the total order $(\prec)$ if
\begin{itemize}
\item $(a,b)$ is a lower wall of $\pi$, and
\item $\ell \prec a$ for all $\ell \in L$, and
\item $b \prec r$ for all $r\in R$.
\end{itemize}
\end{definition}

The group $S_\infty$ acts on the set $\Tot$. More generally, there is an action of the group $\widehat{S}_\infty$ of all bijections $\ZZ\to \ZZ$. If $\pi:\ZZ\to \ZZ$ is a bijection, then $\pi\cdot (\prec)$ is the total order $\prec'$ so that $a\prec b$ if and only if $\pi(a) \prec' \pi(b)$.
There is a straightforward analog of \Cref{prop:wallscoverSn} describing the cover relations in $\WO(\Tot)$, though proving it is more subtle.
\begin{theorem}\label{prop:wallscoverSinfty}
    The following are equivalent, for a total order $(\prec)$ and a pair $(a,b)$ with $a<b$:
    \begin{itemize}
        \item[(a)] The pair $(a,b)$ is a lower wall of $(\prec)$;
        \item[(b)] The product $(a,b)\cdot \pi$ is covered by $(\prec)$ in $\WO(\Tot)$;
        \item[(c)] The pair $(a,b)$ is in $N(\pi)$ and $N(\pi)\setminus\{(a,b)\}$ is the inversion set of some total order.
    \end{itemize}
    Conversely, if $(\prec')\lessdot (\prec)$ is a cover relation, then there exists a unique pair $(a,b)$ so that $N(\pi)\setminus \{(a,b)\} = N(\pi')$.
\end{theorem}
\begin{proof}
    The equivalence between (a), (b), and (c) is straightforward, so we focus on the second statement. Let $(\prec') < (\prec)$ be such that $|N(\prec)\setminus N(\prec')|>1$. We wish to show that there exists a total order $(\prec^s)$ so that $(\prec')<(\prec^s)<(\prec)$. Assume $(a,b),(a',b') \in N(\prec)\setminus N(\prec')$ are distinct pairs. Set $A=\min\{a,a'\}$ and $B=\max\{b,b'\}$ and write $[A,B]_<$ for the interval $\{x\in \ZZ \mid A\leq x \leq B\}$. Then we define a total order $(\prec_{AB})$ by putting $x\prec_{AB} y$ if either $x,y\in [A,B]_<$ and $x\prec y$, or $x,y\not\in [A,B]_<$ and $x<y$. Equivalently, $(\prec_{AB}) $ is the unique total order so that $N(\prec)\cap \{(x,y) \mid A\leq x < y \leq B\} = N(\prec_{AB})$. Define $(\prec'_{AB})$ similarly. We have that $|N(\prec_{AB})\setminus N(\prec'_{AB})| > 1$ since it contains both $(a,b)$ and $(a',b')$. Both $(\prec_{AB})$ and $(\prec_{AB}')$ are in the image of $\WO(S_\infty)$ under the map $\pi\mapsto (\prec_\pi)$, so we deduce by \Cref{prop:wallscoverSn} that there is a total order $(\prec^s_{AB})$ so that $(\prec'_{AB})<(\prec^s_{AB}) < (\prec_{AB})$. Define $(\prec^s)$ to be the join $(\prec')\vee (\prec^s_{AB})$. Note that $(\prec^s_{AB}) < (\prec)$, so by the properties of join we have $(\prec')<(\prec^s) \leq (\prec)$. Hence it remains to show that $(\prec^s)$ differs from $(\prec)$. 
    
    We claim that $N(\prec^s)\cap \{(x,y)\mid A\leq x < y \leq B\} = N(\prec^s_{AB})$. Indeed, if $(x,y) \in N(\prec^s)$, then by \Cref{prop:latticeSinfty} there is a sequence $x=b_0<b_1<\cdots < b_k = y$ so that each $(b_j,b_{j+1})$ is in either $N(\prec')$ or $N(\prec^s_{AB})$. If additionally $x,y\in [A,B]_<$, then each $b_j\in [A,B]_<$. The sets $N(\prec')\cap \{(c,d) \mid A\leq c < d \leq B\}$ and $N(\prec^s_{AB})$ are equal, so each $(b_j,b_{j+1})$ is in $N(\prec^s_{AB})$. This means that $x = b_0 \succ^s_{AB} b_1 \succ^s_{AB} \cdots \succ^s_{AB} b_k = y$, and we see that $(x,y)\in N(\prec^s_{AB})$. But $N(\prec^s_{AB}) \subsetneq N(\prec_{AB})$, so we are done. 
\end{proof}

	This theorem is somewhat surprising, since it is possible that a total order $(\prec)$ has no upper or lower walls at all. 

	\begin{example}
		Consider the total order $(\prec)$ studied in \Cref{ex:joinSinfty}, given by
		\[ \cdots \prec -2 \prec -1 \prec 1 \prec 2 \prec 3 \prec \cdots \prec 0. \]
		Then $(\prec)$ has no lower walls.
		As a result, \Cref{prop:wallscoverSinfty} implies that $(\prec)$ does not cover anything in $\WO(\Tot)$. Yet $(\prec)$ is not the minimal element of $\WO(\Tot)$. The fact that the weak order on total orders is not determined by its cover relations is responsible for a lot of its subtlety.
	\end{example}

	\begin{example}
	Take a bijection of $\ZZ$ with $\QQ$ and let $(\prec)$ be the total order on $\ZZ$ induced by the usual order on $\QQ$. Then every interval $\{x \in \ZZ \mid b \prec x \prec a\}$ is infinite, so the pair $(a,b)$ can never be a lower wall. Similarly, $(\prec)$ has no upper walls. 
	\end{example}

	We record the following lemma, which will be useful later on. 

	\begin{lemma}\label{lem:findcoverfinite}
		Let $(\prec') \leq (\prec)$ in $\WO(\Tot)$. If $(a,b)$ is a pair so that $b\prec a$ and $a\prec' b$, and the interval $[b,a]_{\prec}$ is finite, then there is a total order $(\prec^s)$ so that $(\prec')\leq (\prec^s) \lessdot (\prec)$ and $(\prec^s)\lessdot (\prec)$ is a cover relation.
	\end{lemma}
	\begin{proof}
		By  \Cref{prop:wallscoverSinfty}, it is equivalent to show that there is a lower wall $(x,y)$ of $(\prec)$ so that $(a,b)$ is not in $N(\prec')$. To construct this wall, pick a pair $(x,y) \in N(\prec)\setminus N(\prec')$ which minimizes $|[y,x]_{\prec}|$. Because $(a,b)$ is such a pair and $[b,a]_{\prec}$ is finite, $[y,x]_\prec$ also is finite. In particular, if $y \prec c \prec x$ then $|[y,c]_\prec|$ and $|[c,x]_\prec|$ both have smaller size. 

		We claim that $(x,y)$ is a lower wall. Indeed, we have that $y \prec x$ and that $x \prec' y$. Assume to the contrary that there is some $y \prec c \prec x$. Then either $c \prec' x \prec' y$, $x\prec' c \prec' y$, or $x\prec' y \prec' c$. In the first case, $(y,c) \in N(\prec)\setminus N(\prec')$ contradicts minimality of $(x,y)$. In the second and third case, $(x,c)$ contradicts the minimality of $(x,y)$. Hence $(x,y)$ is a lower wall of $(\prec)$, so we may take $(\prec^s)$ to be the total order $(x,y)\cdot (\prec)$. 
	\end{proof}

\subsection{Canonical join representations for total orders}\label{subsec:canjoinSinfty}
An \newword{arc} for $S_\infty$ is a tuple $(a,b,L,R)$ with $a<b$ and $L\sqcup R = \{a+1,\ldots,b-1\}$. Just like for $S_n$, arcs may be depicted using arc diagrams.
\begin{definition}
Let $(\prec)$ be a total order. A \newword{lower arc} of $(\prec)$ is an arc $(a,b,L,R)$ so that 
\begin{itemize}
	\item $(a,b)$ is a lower wall of $(\prec)$, and
	\item $\ell \prec b$ for all $\ell \in L$, and
	\item $a \prec r$ for all $r\in R$.
\end{itemize}
\end{definition}

We write $\LowerArcs(\prec)$ for the set of lower arcs of $(\prec)$. Unlike for $S_n$, knowledge of $\LowerArcs(\prec)$ does not determine $(\prec)$ uniquely. For example, the total order $(<)$ and the total order from \Cref{ex:joinSinfty} both have no lower arcs.

\begin{example}
Consider the total order $(\prec)$ which is the reverse of the standard ordering of the integers $(<)$:
\[ \cdots \prec 7 \prec 6 \prec 5 \prec 4 \prec 3 \prec 2 \prec 1 \prec 0 \prec \cdots. \]
Then 
\[ \LowerArcs(\prec) = \{ (a,b,\varnothing,\varnothing) \mid a<b \}. \]
The arc diagram for $(\prec)$ is 
\begin{center}
    \begin{tikzpicture}
        \node (A3m) at (-3,0) {$\cdots$};
        \node (A2m) at (-2,0) {$-2$};
        \node (A1m) at (-1,0) {$-1$};
        \node (A0) at (0,0) {$0$};
        \node (A1) at (1,0) {$1$};
        \node (A2) at (2,0) {$2$};
        \node (A3) at (3,0) {$\cdots$};
        \draw (A3m) -- (A2m) -- (A1m) -- (A0) -- (A1) -- (A2) -- (A3);
    \end{tikzpicture}.
\end{center} 
Now let $(\prec') = (1,2)\cdot (\prec)$, which is the order:
\[ \cdots \prec' 7 \prec' 6 \prec' 5 \prec' 4 \prec' 3 \prec' 1 \prec' 2 \prec' 0 \prec' \cdots \]
Then 
\[ \LowerArcs(\prec') = \{ (a,b,\varnothing, \varnothing) \mid a < b,~ (a,b) \neq (0,1)\text{ or }(1,2) \} \cup \{ (0,2,\{1\},\varnothing) \}. \]
The arc diagram for $(\prec')$ is 
\begin{center}
    \begin{tikzpicture}
        \node (A3m) at (-3,0) {$\cdots$};
        \node (A2m) at (-2,0) {$-2$};
        \node (A1m) at (-1,0) {$-1$};
        \node (A0) at (0,0) {$0$};
        \node (A1) at (1,0) {$1$};
        \node (A2) at (2,0) {$2$};
        \node (A3) at (3,0) {$\cdots$};
        \draw (A3m) -- (A2m) -- (A1m) -- (A0)  (A2) -- (A3);
        \draw (A0) to[out=0,in=180] (1,.5) to[out=0,in=180] (A2);
    \end{tikzpicture}.
\end{center}
\end{example}

As with $S_n$, two arcs \newword{cross} if they share an initial value, share a terminal value, or intersect in their interiors in an arc diagram. A \newword{non-crossing collection} is a collection of arcs, no two of which cross. Given an arc $\aa=(a,b,L,R)$, we define $j_\aa$ to be the total order $(\prec)$ defined by
\[ \cdots \prec a-2 \prec a-1 \prec \ell_1 \prec \ell_2 \prec \cdots \prec \ell_{j} \prec b \prec a \prec r_1 \prec \cdots \prec r_{k-1} \prec r_k \prec b+1 \prec b+2 \prec b+3 \prec \cdots. \]
Note that the unique lower arc of $j_\aa$ is $\aa$.

\begin{lemma}\label{lem:arcconeSinfty}
    Let $\aa$ be an arc. If $(\prec)$ is a total order with $\aa\in \LowerArcs(\prec)$, and $(\prec')$ is a total order with $(a,b)\in N(\prec')$ and $(\prec')\leq (\prec)$, then $j_\aa \leq (\prec')$.  
\end{lemma}
\begin{proof}
    Let $\aa = (a,b,L,R)$. To show that $j_{\aa}\leq (\prec')$, it is enough to show that $b\prec' a$, $\ell\prec' a$ for each $\ell\in L$, and $b\prec' r$ for each $r\in R$. Because $\aa\in \LowerArcs(\prec)$ and $(\prec')\leq (\prec)$, we have $b\prec' a$, $\ell \prec' b$ for each $\ell\in L$, and $a\prec' r$ for each $r\in R$. Hence the claim follows by transitivity of $\prec'$. 
\end{proof}

\begin{lemma}\label{lem:joinofarcsSinfty}
    Let $D$ be any non-crossing collection of arcs. Then 
    \[ \LowerArcs(\bigvee_{\aa\in D} j_\aa) = D. \] 
    Furthermore, if $(\prec)$ is any total order with $D\subseteq \LowerArcs(\prec)$, then
    $\bigvee_{\aa\in D} j_\aa \leq (\prec)$. 
\end{lemma}
\begin{proof}
    First assume $D$ is a \emph{finite} non-crossing collection of arcs. Then if $A$ is the minimum initial value among the arcs in $D$ and $B$ is the maximum terminal value among the arcs in $D$, then there is a permutation $\pi \in S_\infty$ so that $\bigvee_{\aa\in D} j_\aa = (\prec_\pi)$ and $\pi$ fixes all integers outside the interval $[A,B]_{<}$. Furthermore, $\pi$ is the join of the join-irreducible permutations in $S_{B-A}$ with lower arcs in $D$. Hence by \Cref{prop:noncrossingSn,prop:canonicaljoinSn}, the lower arcs of $\pi$ coincide with $D$, and therefore $\LowerArcs(\prec_\pi) = D$.

    Now, if $D$ is infinite, then enumerate the arcs in $D$ by $\aa_1,\aa_2,\aa_3,\ldots $. Set $(\prec_k) = \bigvee_{i=1}^kj_{\aa_i}$. Set $(\prec) = \bigvee_{k=1}^\infty (\prec_k)$. Since this is the join of an increasing chain, 
    $N(\prec) = \bigcup_{k=1}^\infty N(\prec_k)$. Let $\aa_k=(a,b,L,R)\in D$. Then $(a,b)\in N(\prec)$ since $(a,b)\in N(\prec_k)$. Let $c\in\ZZ$, we will check $c \not\in [b,a]_{\prec}$ so that $(a,b)$ is a lower wall of $(\prec)$. If $c<a$ and $b\prec c$, then there is some $k'\geq k$ so that $(b,c)\in N(\prec_{k'})$. But then $(a,b) \in N(\prec_{k'})$ as well or else we contradict that $(a,b)$ is a lower wall of $(\prec_{k'})$. Hence $a\prec c$, so $c\not\in [b,a]_\prec$. The cases where $a < c < b$ and $b<c$ are similar. We conclude that $(a,b)$ is a lower wall of $(\prec)$. We must also check that the lower arc of $(\prec)$ from $a$ to $b$ coincides with $\aa_k$. Again, this can only fail by picking up an inversion in some $(\prec_{k'})$ which forces the arc from $a$ to $b$ in $(\prec_k)$  to differ from $\aa_k$, contradicting the first paragraph. Hence the lower arc of $(\prec)$ from $a$ to $b$ is $\aa_k$, so we find that $D\subseteq \LowerArcs(\prec)$. Conversely, if $(a,b)$ is some lower wall of $(\prec)$ then there is some $k$ for which $(a,b)\in N(\prec_k)$. The interval $[b,a]_{\prec_k}$ must be finite, since $N(\prec_k)$ is finite. For any $c \in [b,a]_{\prec_k}$, we must have either $c<a$ or $c>b$ (or else $(a,b),(b,c)\in N(\prec_k)\subseteq N(\prec)$, which is impossible). If $c<a$ there must be some $k_c\geq k$ so that $(c,a)\in N(\prec_{k'})$ and if $c>b$ then there must be some $k_c\geq k$ so that $(b,c)\in N(\prec_{k'})$. Letting $K$ be the maximum of $\{k_c \mid c\in [b,a]_{\prec_k}\}$, we find that $(a,b)$ is a lower wall of $(\prec_K)$, which means that it is represented among the arcs $\aa_1,\ldots,\aa_K$. We conclude that $\LowerArcs(\prec)= D$.

    The second claim follows from \Cref{lem:arcconeSinfty}, which implies that $j_\aa\leq (\prec)$ whenever $(\prec)$ has $\aa$ as a lower arc.
\end{proof}

Before we state the analog of \Cref{prop:finitecanonicaljoin} in this setting, we need one more definition.

\begin{definition}\label{def:widelygenerated}
We say an element $x$ of a poset $P$ is \newword{widely generated} if for all $y<x$ there exists a cover relation $x'\lessdot x$ so that $y\leq x'$. 
\end{definition}
\begin{aside}
\begin{remark}
The term \emph{widely generated} arises in the lattice of torsion classes in a finite-length abelian category, where the widely generated torsion classes are exactly those that are generated by wide subcategories \cite{wide}.
\end{remark}
\end{aside}


\begin{theorem}\label{prop:noncrossingSinfty}
If $(\prec)$ is a total order, then $\LowerArcs(\prec)$ is a non-crossing collection. The map $(\prec) \mapsto \LowerArcs(\prec)$ restricts to a bijection from widely generated total orders to non-crossing collections.
The inverse sends a non-crossing collection $D$ to $\bigvee_{\aa\in D} j_\aa$.
\end{theorem}
\begin{proof}
    Let $(\prec)$ be a total order. Consider two lower arcs $(a,b,L,R)$ and $(a',b',L',R')$ of $(\prec)$. The restriction of $\prec$ to a total ordering on the (finite) set $[\min\{a,a'\}, \max\{b,b'\}]_{<}$ comes from a permutation $\pi$ of that set. Further, the lower arcs of $\pi$ include $(a,b,L,R)$ and $(a',b',L',R')$. Hence the arcs do not cross by \Cref{prop:finitecanonicaljoin}. We conclude that $\LowerArcs(\prec)$ is a non-crossing collection. 

    If $D$ is a non-crossing collection, then consider $(\prec) = \bigvee_{\aa\in D} j_\aa$. By \Cref{lem:joinofarcsSinfty}, $\LowerArcs(\prec) = D$. Consider some $(\prec')<(\prec)$. Then there exists $\aa=(a,b,L,R)\in D$ so that $(\prec') \not \geq j_\aa$. By \Cref{lem:arcconeSinfty}, this implies $(a,b) \not \in N(\prec')$. Hence $(\prec') \leq (a,b)\cdot (\prec) \lessdot (\prec)$, so $(\prec)$ is widely generated.

    Finally, let $(\prec)$ is widely generated and $(\prec') = \bigvee_{\aa \in \LowerArcs(\prec)} j_\aa$. By \Cref{lem:joinofarcsSinfty}, $(\prec')\leq (\prec)$. Assume to the contrary that $(\prec')< (\prec)$. Then there is some lower wall $(a,b)$ of $(\prec)$ which is not in $N(\prec')$, since $(\prec)$ is widely generated. But this contradicts \Cref{lem:joinofarcsSinfty}, which says that $(a,b)$ is a lower wall of $(\prec')$. 

    We deduce that $(\prec) \mapsto \LowerArcs(\prec)$ and $D\mapsto \bigvee_{\aa\in D}j_\aa$ are inverse bijections between the widely generated total orders and the noncrossing collections.
\end{proof}

\begin{aside}
\begin{remark}
A consequence of \Cref{prop:noncrossingSinfty} is that $(\prec)$ is widely generated if and only if $(\prec)$ is of the form $\bigvee_{\aa \in D}j_\aa$ for some non-crossing collection $D$. It seems difficult to give a more explicit characterization of the widely generated total orders. For example, the total orders
\[ 1 \prec -1 \prec 3 \prec -3 \prec 5 \prec -5 \prec \cdots \prec 2\cdot 1 \prec 2\cdot -1 \prec 2\cdot 3 \prec \cdots \prec 2^2\cdot 1 \prec 2^2\cdot -1 \prec 2^2 \cdot 3 \prec \cdots \prec 0 \]
and 
\[ \cdots \prec -1 \prec 1 \prec 3 \prec 5 \prec \cdots \prec 6 \prec 4 \prec 2 \prec 0 \prec \cdots \]
are widely generated, while the total order
\[ \cdots \prec -1 \prec 1 \prec 3 \prec 5 \prec \cdots \prec 0 \prec 2 \prec 4 \prec 6 \prec \cdots \]
is not.
\end{remark}
\end{aside}

In a finite lattice, the join-irreducibles are the same as the elements which have a unique lower cover. The following is the generalization of this fact to an infinite complete lattice. Its proof is straightforward.
\begin{lemma}\label{lem:JIwide}
    Let $L$ be a complete lattice. Then $j\in L$ is a complete join-irreducible if and only if $j$ is widely generated and there is a unique element $j_*\in L$ so that $j_*\lessdot j$ is a cover.
\end{lemma}

Hence \Cref{prop:noncrossingSinfty} implies that there is a bijection between arcs and \JIs. The unique \JI with lower arc $\aa=(a,b,L,R)$ is exactly the total order $j_\aa$ described above.

We have shown:
\begin{corollary}
    The map $\aa \mapsto j_\aa$ is a bijection between arcs and $\JIrrc(\WO(\Tot))$.
\end{corollary}

Unlike in the case of $S_n$, there are elements of $\WO(\Tot)$ without a canonical join representation. Indeed, we have the following.
\begin{lemma}\label{lem:canonicaljoiniswide}
    If $L$ is a complete lattice and $y\in L$ is not widely generated, then $y$ has no canonical join representation.
\end{lemma}
\begin{proof}
    Let $y\in L$ have canonical join representation $y = \bigvee U$, where $U=\{j_i\}_{i\in I}$.  We will show that $y$ is widely generated. Consider some $x\in L$ with $x<y$. Pick some $i\in I$ so that $j_i \not \leq x$. (Such an $i$ must exist, else $y\leq x$.) Consider the set $V = \{ y' < y \mid j_i \not\leq y'\}$. Then $x\in V$ and $j_{i'} \in V$ for all $i'\neq i$. It must be that $\bigvee V < y$, because otherwise $\bigvee V=y$ is a join representation so that $U \not \ll V$ which is impossible. In fact, $\bigvee V$ is covered by $y$, since for any $y'$ with $\bigvee V < y' \leq y$ we have that $j_i\leq y'$, implying that $\bigvee U = y \leq y'$ so $y=y'$. We deduce that $x \leq \bigvee V \lessdot y$. Hence $y$ is widely generated.
\end{proof}


\begin{theorem}\label{thm:canonicaljoinSinfty}
    Let $(\prec) \in \WO(\Tot)$. Then $(\prec)$ has a canonical join representation if and only if $(\prec)$ is widely generated. In this case, the canonical join representation of $(\prec)$ is 
    \[ (\prec) = \bigvee_{\aa\in \LowerArcs(\prec)} j_\aa . \]
\end{theorem}
\begin{proof}
    By \Cref{lem:canonicaljoiniswide}, a total order with a canonical join representation must be widely generated. So let $(\prec)$ be a widely generated total order. By \Cref{prop:noncrossingSinfty}, we have $(\prec) = \bigvee_{\aa \in \LowerArcs(\prec)}j_\aa$. We wish to show this is a canonical join representation. We see that it is irredundant since for any $D\subsetneq \LowerArcs(\prec)$, the join $\bigvee_{\aa\in D} j_\aa$ has $D\neq \LowerArcs(\prec)$ as its set of lower walls (using \Cref{thm:noncrossingarcsTITO}). Now let $V\subseteq L$ be so that $\bigvee V = (\prec)$. For any $\aa=(a,b,L,R)\in \LowerArcs(\prec)$, we wish to show that $j_\aa \leq v$ for some $v\in V$. If any $v\in V$ has $(a,b)$ as an inversion then we are done, since \Cref{lem:arcconeSinfty} says that $j_\aa \leq v$. Otherwise, no $v\in V$ has $(a,b)$ as an inversion. But since $(a,b) \in N(\prec) = \overline{\bigcup_{v\in V} N(v)}$, there must be some sequence $b_1,\ldots, b_k$ with $a=b_0<b_1 < \cdots < b_k < b_{k+1}=b$ and $v_0,\ldots,v_k\in V$ so that $(b_i,b_{i+1})\in N(v_i)$ for each $i$. But this would mean in particular that $(a,b_1)$ and $(b_1,b)$ are both in $N(\prec)$, so $b\prec b_1 \prec a$, contradicting that $(a,b)$ is a lower wall of $(\prec)$. Hence this case doesn't happen and $\{j_\aa \mid \aa\in \LowerArcs(\prec)\} \ll V$. 
\end{proof}

\section{Translation-invariant total orders}\label{sec:TITO}

Throughout this section we fix a natural number $n$. We now introduce the second main object of interest in this paper.
\begin{definition}
    A \newword{translation-invariant total order} (\newword{TITO}) is a total order $(\prec)$ of the integers such that for all $a,b\in \ZZ$, we have $a\prec b$ if and only if $a+n \prec b+n$. We write $\TTot_n$ for the set of TITOs.
\end{definition}

\begin{example}\label{ex:threetitos}
Here are three examples of TITOs with $n=4$ that we will visit again later.

\begin{equation}\cdots\prec-3\prec-4\prec -1\prec -2 \prec 1\prec 0 \prec 3 \prec 2\prec 5 \prec 4 \prec 7 \prec 6 \prec \cdots,
\label{eq:tito1}\end{equation}
    \begin{equation} \cdots -3 \prec 1 \prec 5 \prec 9 \prec \cdots \prec 6 \prec 7 \prec 2 \prec 3 \prec -2 \prec -1 \prec \cdots \prec -4 \prec 0 \prec 4 \prec 8 \prec \cdots, \label{eq:tito2}\end{equation}
    \begin{equation} \cdots\prec -3\prec -1\prec 1\prec 3\prec 5\prec 7\prec \cdots 
    \prec 6\prec 4\prec 2\prec 0\prec -2\prec -4\prec \cdots.  \label{eq:tito3}\end{equation}
\end{example}

We will now examine the structure of a TITO as an ordering of $\ZZ$. 
Let $(\prec)$ be a total order on $\ZZ$. We say that a subset $I\subseteq \ZZ$ is \newword{order-convex} if for any $a\prec b\prec c$ such that $a$ and $c$ are in $I$, then also $b$ is in $I$. For example, order ideals and intervals are order-convex. 

\begin{definition}
    Let $(\prec)$ be a TITO. A \newword{block} of $(\prec)$ is an order-convex subset $I$ with the following properties:
    \begin{itemize}
        \item The ordering of $I$ by $(\prec)$ has no minimal or maximal element, and
        \item For any $a,c\in I$, the interval $\{b\in I\mid a\prec b \prec c\}$ is finite.
    \end{itemize}
    The \newword{size} of a block $I$ is the number of residue classes modulo $n$ appearing in $I$. We say that $I$ is a \newword{waxing} block if $x\prec x+n$ for all $x\in I$. We say that $I$ is a \newword{waning} block if $x+n \prec x$ for all $x\in I$. 
\end{definition}

For example, the TITO (\ref{eq:tito1}) has one block, which is waxing. The TITO (\ref{eq:tito2}) has three blocks which, from left to right, are waxing, waning, and waxing.

\begin{lemma}\label{lem:fundTITO}
    Let $(\prec)$ be a TITO. Then the following hold:
    \begin{itemize}
        \item[(a)] Every integer is in a unique block.
        \item[(b)] If $a\equiv b\bmod n$, then $a$ and $b$ are in the same block.
        \item[(c)] Every block is either waxing or waning.
        \item[(d)] If $x_1\prec \cdots \prec x_k \prec x_{k+1}$ is an interval in a block of size $k$, then the residue classes of $x_1,\ldots, x_k$ are distinct, and $x_{k+1} = x_1 + n$ or $x_1 - n$ depending on whether the block is waxing or waning, respectively.
    \end{itemize}
\end{lemma}
\begin{proof}
    First, we observe that translation-invariance implies that either $\cdots\prec b-n \prec b \prec b+n \prec \cdots$ or $\cdots\prec b+n \prec b \prec b-n \prec \cdots$ for each integer $b$. 
    
    Proof of (b): The claim will follow if we can show that the interval between $a$ and $a+n$ is finite. 
    Suppose not. Then there are integers $b,b'$ between $a$ and $a+n$ which are congruent modulo $n$. Assume $a\prec a+n$, the other case is dual. Then without loss of generality, we may take $a\prec b \prec b' \prec a+n$. By translation-invariance, also $a+n\prec b+n$ and $b-n\prec a$. We conclude that $b-n \prec b \prec b' \prec b+n$, so there are two integers congruent to $b$ between $b-n$ and $b+n$. However, translation-invariance implies that the residue class of $b$ is in either increasing or decreasing order, a contradiction.

    Proof of (a): Each integer is in at most one block because the block containing $a$ must be the collection of integers $I=\{x \mid \text{the interval between $x$ and $a$ is finite}\}$. So existence of the block is equivalent to showing that $I$ has no maximum or minimum, which follows from (b). 

    Proof of (d): Assume $a\prec a+n$ and $a$ is in a block of size $k$ (the other case is dual). It suffices to show that if the interval between $a$ and $a+n$ is $a\prec x_2 \prec \cdots \prec x_\ell \prec a+n$, then the residue classes of $x_2,\ldots, x_\ell$ are distinct and $\ell=k$. Indeed, the proof (b) establishes that the residue classes are distinct. So suppose $\ell<k$. Let $x$ be the rightmost element in the block which is to the left of $a$ and which is not congruent to any of $x_1,\ldots,x_\ell$. Then $x \prec a$, so $x+n\prec a+n$. By assumption, this implies $x+n\prec a$. Also, $x\prec a+n$, so $x-n \prec a$. However, at least one of $x+n$ or $x-n$ must be to the right of $x$, contradicting our hypothesis. This proves (d).

    Proof of (c): Assume $a \prec a+n$ (the other case is dual) and that $b$ is in the same block as $a$. Then, by (d), there is some $b'$ congruent to $b$ so that $a\prec b' \prec a+n$. By translation-invariance, $a+n\prec b'+n$, so $b' \prec b'+n$ and thus the residue class of $b$ is in increasing order. The claim follows.
\end{proof}

According to \Cref{lem:fundTITO}, the blocks partition $\ZZ$ in such a way that congruent integers are in the same part. Furthermore, \Cref{lem:fundTITO}(c+d) implies that each block is encoded by a finite amount of data; namely, whether the block is waxing or waning and a list of $k$ consecutive elements of the block, where $k$ is the size of the block. By iteratively applying \Cref{lem:fundTITO}(d) to this list, we can rebuild the entire block. We encode this data with a \newword{window}: a window for a size $k$ block $I$ is a list of $k$ consecutive entries $x_1,\ldots,x_k$ from the block. If $I$ is waxing, then we depict the window by $[x_1,\ldots,x_k]$. If $I$ is waning, then we depict the window with an underline by $[\underline{x_1,\ldots,x_k}]$. A \newword{window notation} for the TITO $(\prec)$ consists of a window for each of its blocks, ordered left to right in the same order that they appear in $(\prec)$.

\begin{example}
    The following are window notations for the three TITOs of \Cref{ex:threetitos}.
    \begin{equation*}
        [1,0,3,2]
        \tag{1}
    \end{equation*}
    \begin{equation*}
        [1][\underline{2,3}][0]
        \tag{2}
    \end{equation*}
    \begin{equation*}
        [1,3][\underline{2,0}]
        \tag{3}
    \end{equation*}
    We note that the choice of window for a given block is not unique. For instance, consider the TITO with $n=4$
\begin{equation*}
    \cdots \prec -1 \prec 2 \prec 1 \prec 3 \prec  \cdots \prec 8 \prec 4 \prec 0 \prec -4 \prec \cdots.
\end{equation*}
Then valid window notations for this TITO include 
\[ [2,1,3][\underline{0}], \qquad [-1,2,1][\underline{4}], \qquad [3,5,4][\underline{-64}].  \]
The different windows for a block may be constructed by \emph{sliding} the window: if $[x_1,\ldots,x_k]$ is a window for a waxing block, then so are $[x_2,\ldots,x_k,x_1+n]$ and $[x_k-n,x_1,\ldots,x_{k-1}]$. Similarly, if $[\underline{x_1,\ldots,x_k}]$ is a window for a waning block, then so are $[\underline{x_2,\ldots,x_k,x_1-n}]$ and $[\underline{x_k+n,x_1,\ldots,x_{k-1}}]$.
\end{example}

Recall that we also discussed window notations for elements of the affine symmetric group $\tS_n$. Given an affine permutation $\tpi$, we let $(\prec_{\tpi})$ be the TITO with the same window notation as $\tpi$. More precisely, $a\prec_{\tpi} b$ if and only if $\tpi^{-1}(a) < \tpi^{-1}(b)$. For example, TITO (\ref{eq:tito1}) is $(\prec_{\tpi})$ for the affine permutation $\tpi$ sending $1\mapsto 0$, $2\mapsto 3$, $3\mapsto 2$, $4\mapsto 5$, etc.

There is a group action of $\tS_n$ on $\TTot_n$. If $(\prec)$ is a TITO and $\tpi$ is an affine permutation, then $\tpi\cdot (\prec)$ is the TITO $(\prec')$ so that $a \prec b$ if and only if $\tpi(a) \prec' \tpi(b)$. In particular, the TITO $(\prec_{\tpi})$ coincides with $\tpi\cdot (<)$.

\subsection{The weak order on TITOs}
In this section, we describe a partial order on TITOs.
The resulting poset was first studied in \cite{Barkley2022}. To define it, we need the equivalence relation on pairs $(a,b)$ putting $(a,b)\sim (a+n, b+n)$. The equivalence class containing $(a,b)$ is denoted $\rind{a,b}$. A \newword{reflection index} is an equivalence class $\rind{a,b}$ with $a<b$. We will freely extract $a$ and $b$ from $\rind{a,b}$, but only when the result is independent of the choice of representative pair $(a,b)$.  
\begin{definition}\label{def:TITO}
    An \newword{inversion} of a TITO $(\prec)$ is a reflection index $\rind{a,b}$ so that $b\prec a$. Let $\tN(\prec)$ denote the set of inversions of $(\prec)$. The partial order on TITOs putting $(\prec_1)\leq (\prec_2)$ if $\tN(\prec_1)\subseteq \tN(\prec_2)$ is called the \newword{weak order} on TITOs. The resulting poset is denoted $\WO(\TTot_n)$.
\end{definition}
The poset $\WO(\TTot_2)$ is shown in \Cref{fig:S2tilde}.
Note that $N(\prec)$ is a set of pairs, while $\tN(\prec)$ is a set of reflection indices. The word \emph{inversion} could refer to an element of either, depending on the context. 

\begin{example}
    The TITO $(<)$, which is the usual order on the integers, has inversion set $\tN(<) = \varnothing$.
    The TITO $(\prec)$ shown in (\ref{eq:tito1}), with window notation $[1,0,3,2]$, has inversion set
    \[ \tN(\prec) = \{\rind{0,1}, \rind{2,3} \}.\]
    The TITO $(\prec)$ with window notation $[\underline{1}][2]$ (and $n=2$) has inversion set
    \[ \tN(\prec) = \{\rind{1,3},\rind{1,5},\rind{1,7},\ldots \} \cup \{\rind{2,3},\rind{2,5},\rind{2,7},\ldots\}. \]
\end{example}

From \Cref{def:TITO}, $\WO(\TTot_n)$ is evidently a subposet of $\WO(\Tot)$. More is true.

\begin{theorem}\label{thm:latticetSn}
    $\WO(\TTot_n)$ is a complete sublattice of $\WO(\Tot)$.
\end{theorem}
\begin{proof}
    Let $(\prec_i)_{i\in I}$ be a family of TITOs. By \Cref{prop:latticeSinfty}, their join in $\WO(\Tot)$ is the total order $(\prec)$ with
    \[ N(\prec) = \overline{\bigcup_{i\in I} (\prec_i)}. \]
    We must check that $(\prec)$ is translation-invariant. Let $a<b$. Then $b\prec a$ if and only if $(a,b)\in N(\prec)$. By definition of closure, this occurs if and only if there is a sequence $b_0,b_1,\ldots,b_k$ with $a=b_0 < b_1 < \cdots < b_k=b$ so that each pair $(b_{j},b_{j+1})$ is an inversion of $(\prec_{i_j})$ for some index $i_j\in I$. Since each $(\prec_{i_j})$ is translation-invariant, this occurs if and only if there is a sequence  $b_0,b_1,\ldots,b_k$ with $a=b_0 < b_1 < \cdots < b_k=b$ so that each pair $(b_{j}+n,b_{j+1}+n)$ is an inversion of $(\prec_{i_j})$ for some index $i_j\in I$. This is equivalent to $b+n\prec a+n$, so we are done.
    The proof that a meet of TITOs is a TITO is similar.
\end{proof}

Let $\tT$ denote the set of reflection indices $\{\rind{a,b} \mid a<b\}$. Then we can define a closure operator on subsets of $\tT$ by sending $X\subseteq \tT$ to the minimal set $\overline{X}$ such that for any $a<b<c$, if $\rind{a,b},\rind{b,c} \in \overline{X}$, then $\rind{a,c}\in \overline{X}$. Then \Cref{thm:latticetSn} shows that the join of a family of TITOs $(\prec_i)_{i\in I}$ is the unique TITO $(\prec)$ so that
\[ \tN(\prec) = \overline{\bigcup_{i\in I} \tN(\prec_i)}. \]

\begin{figure}
\centering 
	\begin{tikzpicture}
		

		\draw[fill] (-90:2.3) circle (1.2pt) -- ({.5*atan2(0,0+1)+.5*atan2(1,1+1)-45}:2.3) circle (1.2pt)
		-- ({.5*atan2(1,1+1)+.5*atan2(2,2+1)-45}:2.3) circle (1.2pt)
		-- ({.5*atan2(2,2+1)+.5*atan2(3,3+1)-45}:2.3) circle (1.2pt)
		-- ({.5*atan2(3,3+1)+.5*atan2(4,4+1)-45}:2.3) circle (.8pt)
		-- ({.5*atan2(4,4+1)+.5*atan2(5,5+1)-45}:2.3) circle (.6pt);

        \draw (-90:2.7) node {$[1,2]$};
        \draw ({.5*atan2(0,0+1)+.5*atan2(1,1+1)-45}:2.8) node[scale=.8] {$[2,1]$};
        \draw ({.5*atan2(1,1+1)+.5*atan2(2,2+1)-48}:2.7) node[scale=.6] {$[-1,4]$};
        \draw ({.5*atan2(2,2+1)+.5*atan2(3,3+1)-45}:2.7) node[scale=.5] {$[4,-1]$};
        
		\foreach \x in {1,2,...,3} {
			\fill ({atan2(4.5,4.5+1)-45+1.4668*\x}:2.3) circle (.6pt);
		};

		\draw[fill] (-90:2.3) circle (1.2pt) -- ({.5*atan2(0+1,0)+.5*atan2(1+1,1)-45}:-2.3) circle (1.2pt)
		-- ({.5*atan2(1+1,1)+.5*atan2(2+1,2)-45}:-2.3) circle (1.2pt)
		-- ({.5*atan2(2+1,2)+.5*atan2(3+1,3)-45}:-2.3) circle (1.2pt)
		-- ({.5*atan2(3+1,3)+.5*atan2(4+1,4)-45}:-2.3) circle (.8pt)
		-- ({.5*atan2(4+1,4)+.5*atan2(5+1,5)-45}:-2.3) circle (.6pt);

        \draw ({.5*atan2(0+1,0)+.5*atan2(1+1,1)-45}:-2.8) node[scale=.8] {$[0,3]$};
        \draw ({.5*atan2(1+1,1)+.5*atan2(2+1,2)-42}:-2.65) node[scale=.6] {$[3,0]$};
        \draw ({.5*atan2(2+1,2)+.5*atan2(3+1,3)-45}:-2.7) node[scale=.5] {$[-2,5]$};

		\foreach \x in {1,2,...,3} {
			\fill ({atan2(4.5+1,4.5)-45-1.4668*\x}:-2.3) circle (.6pt);
		};

            \draw (0:2.6) node[scale=.5] {$[2][1]$};
        \draw (0:-2.6) node[scale=.5] {$[1][2]$};

			\draw[fill] (0:2.3) circle (1.2pt);
			\draw[fill] (0:-2.3) circle (1.2pt);
            
        \draw[fill] (-2.3,0) -- (-1.9,.5) circle (1.2pt) -- (-2.3,1);
        \draw[fill] (-2.3,0) -- (-2.7,.5) circle (1.2pt) -- (-2.3,1);
        \draw[fill] (2.3,0) -- (1.9,.5) circle (1.2pt) -- (2.3,1);
        \draw[fill] (2.3,0) -- (2.7,.5) circle (1.2pt) -- (2.3,1);

        \draw (-1.6,.5) node[scale=.5] {$[1][\underline{2}]$};
        \draw (-3,.5) node[scale=.5] {$[\underline{1}][2]$};
        \draw (3,.5) node[scale=.5] {$[2][\underline{1}]$};
        \draw (1.6,.5) node[scale=.5] {$[\underline{2}][1]$};
        
		\begin{scope}[yscale=-1,shift={(0,-1)}]
			\draw[fill] (-90:2.3) circle (1.2pt) -- ({.5*atan2(0,0+1)+.5*atan2(1,1+1)-45}:2.3) circle (1.2pt)
			-- ({.5*atan2(1,1+1)+.5*atan2(2,2+1)-45}:2.3) circle (1.2pt)
			-- ({.5*atan2(2,2+1)+.5*atan2(3,3+1)-45}:2.3) circle (1.2pt)
			-- ({.5*atan2(3,3+1)+.5*atan2(4,4+1)-45}:2.3) circle (.8pt)
			-- ({.5*atan2(4,4+1)+.5*atan2(5,5+1)-45}:2.3) circle (.6pt);

            \draw (0:2.6) node[scale=.5] {$[\underline{2}] [\underline{1}]$};
        \draw (0:-2.6) node[scale=.5] {$[\underline{1}] [\underline{2}]$};

        \draw (-90:2.7) node {$[\underline{2,1}]$};
        \draw ({.5*atan2(0,0+1)+.5*atan2(1,1+1)-45}:2.8) node[scale=.8] {$[\underline{3,0}]$};
        \draw ({.5*atan2(1,1+1)+.5*atan2(2,2+1)-48}:2.65) node[scale=.6] {$[\underline{0,3}]$};
        \draw ({.5*atan2(2,2+1)+.5*atan2(3,3+1)-45}:2.7) node[scale=.5] {$[\underline{5,-2}]$};

        \draw ({.5*atan2(0+1,0)+.5*atan2(1+1,1)-45}:-2.8) node[scale=.8] {$[\underline{1,2}]$};
        \draw ({.5*atan2(1+1,1)+.5*atan2(2+1,2)-42}:-2.7) node[scale=.6] {$[\underline{4,-1}]$};
        \draw ({.5*atan2(2+1,2)+.5*atan2(3+1,3)-45}:-2.7) node[scale=.5] {$[\underline{-1,4}]$};
   
			\foreach \x in {1,2,...,3} {
				\fill ({atan2(4.5,4.5+1)-45+1.4668*\x}:2.3) circle (.6pt);
			};

			\draw[fill] (-90:2.3) circle (1.2pt) -- ({.5*atan2(0+1,0)+.5*atan2(1+1,1)-45}:-2.3) circle (1.2pt)
			-- ({.5*atan2(1+1,1)+.5*atan2(2+1,2)-45}:-2.3) circle (1.2pt)
			-- ({.5*atan2(2+1,2)+.5*atan2(3+1,3)-45}:-2.3) circle (1.2pt)
			-- ({.5*atan2(3+1,3)+.5*atan2(4+1,4)-45}:-2.3) circle (.8pt)
			-- ({.5*atan2(4+1,4)+.5*atan2(5+1,5)-45}:-2.3) circle (.6pt);

			\foreach \x in {1,2,...,3} {
				\fill ({atan2(4.5+1,4.5)-45-1.4668*\x}:-2.3) circle (.6pt);
			};

			\draw[fill] (0:2.3) circle (1.2pt);
			\draw[fill] (0:-2.3) circle (1.2pt);
		\end{scope}

        \draw[thick] (-4, -3.4) rectangle (4,4.4);
	\end{tikzpicture}
	\caption{The Hasse diagram for $\WO(\TTot_2)$. Each TITO is labeled by its window notation.}
	\label{fig:S2tilde}
\end{figure}
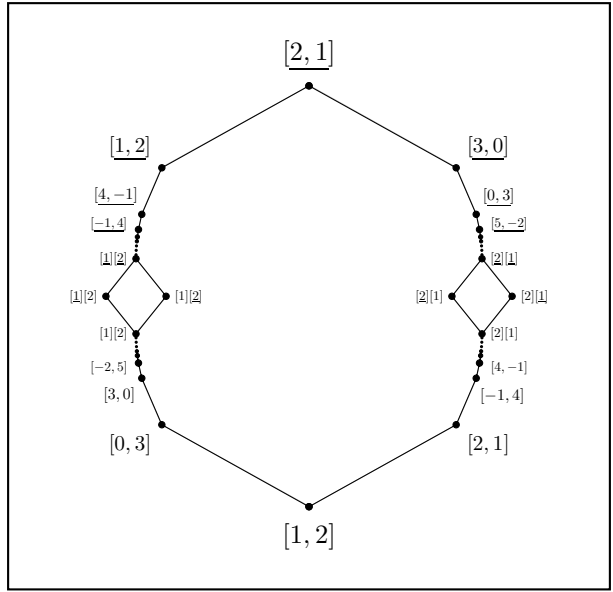

    

\subsection{Walls and cover relations}

Just like in $\WO(S_n)$, the cover relations in $\WO(\TTot_n)$ are governed by adjacent pairs in a TITO. We say that a reflection index $\rind{a,b}$ is \newword{imaginary} if $a\equiv b\bmod n$, otherwise it is \newword{real}. We say a reflection index $\rind{a,b}$ is a \newword{lower wall} of $(\prec)$ if $b\precdot a$ is a cover relation. Note that the imaginary reflection indices $\rind{a,a+kn}$ with $k\geq 2$ are never lower walls. If $\rind{a,b}$ is a real reflection index, then we write $t_{ab}$ to denote the affine permutation which acts on the integers via
\[ t_{ab}: i \mapsto \begin{cases}
    i + b-a &\text{if $i\equiv a\mod n$} \\
    i + a-b &\text{if $i\equiv b\mod n$} \\
    i &\text{otherwise}
\end{cases}. \]
In other words, $t_{ab}$ swaps $a$ and $b$, $a+n$ and $b+n$, $a+2n$ and $b+2n$, etc. 

\begin{theorem}\label{thm:wallscovertSn}
    The following are equivalent, for a TITO $(\prec)$ and a real reflection index $\rind{a,b}$:
    \begin{itemize}
        \item[(a)] The reflection index $\rind{a,b}$ is a lower wall of $(\prec)$;
        \item[(b)] The TITO $t_{ab}\cdot (\prec)$ is covered by $(\prec)$;
        \item[(c)] The reflection index $\rind{a,b}$ is in $\tN(\prec)$ and $\tN(\prec)\setminus \{\rind{a,b}\}$ is the inversion set of some TITO (necessarily covered by $(\prec)$).
    \end{itemize}
    The following are equivalent, for a TITO $(\prec)$ and the imaginary reflection index $\rind{a,a+n}$:
    \begin{itemize}
        \item[(d)] The reflection index $\rind{a,a+n}$ is a lower wall of $(\prec)$;
        \item[(e)] The block of $(\prec)$ containing $a$ is waning of size $1$;
        \item[(f)] The reflection index $\rind{a,a+n}$ is in $\tN(\prec)$ and $\tN(\prec)\setminus \{\rind{a,a+kn} \mid k\geq 1\}$ is the inversion set of some TITO (necessarily covered by $(\prec)$).
    \end{itemize}
    Conversely, if $(\prec_1) \lessdot (\prec_2)$ is a cover relation in $\WO(\TTot_n)$, then either there exists a unique real reflection index $\rind{a,b}$ so that $\tN(\prec_2)\setminus \{\rind{a,b}\} = \tN(\prec_1)$, or else there is a unique imaginary reflection index $\rind{a,a+n}$ so that $\tN(\prec_2) \setminus \{\rind{a,a+kn}\mid k\geq 1\} = \tN(\prec_1)$, but not both.
\end{theorem}
\begin{proof} 
    We first prove that if $(\prec_1) < (\prec_2)$, then there is either a lower wall $\rind{a,b}$ of $(\prec_2)$ so that $\rind{a,b}$ is not an inversion of $(\prec_1)$ or there is an upper wall $\rind{a,b}$ of $(\prec_1)$ which is an inversion of $(\prec_2)$.  To see this, consider $(\prec_1),(\prec_2)$ as elements of $\WO(\Tot)$. There is some inversion $(a,b) \in N(\prec_2)$ which is not an inversion of $(\prec_1)$. If $[b,a]_{\prec_2}$ is finite, then we are done by \Cref{lem:findcoverfinite}. Otherwise, $b$ and $a$ are in different blocks of $(\prec_2)$.
    Now, examining the order that $(\prec_2)$ induces on the set $(a+n\ZZ) \cup (b+n\ZZ)$ (which, after renaming $a$ to $1$ and $\min\{b'\mid b'\equiv b\bmod n,~a<b'\}$ to $2$, must coincide with one of the TITOs in \Cref{fig:DyertS2}), we find that the only way we could have both $b\prec_2 a$ in different blocks and $a\prec_1 b$ is if $a$ and $b$ are in the same block of $(\prec_1)$. But this means that if we apply \Cref{lem:findcoverfinite} to the reverse of $(\prec_1)$, then there is some upper wall $(c,d)$ of $(\prec_1)$ so that $(c,d)\in N(\prec_2)$. Equivalently, $\rind{c,d}$ is an upper wall of $(\prec_1)$ which is also in $\tN(\prec_2)$. 

    (a)$\implies$(b+c): 
    If $\rind{a,b}$ is a lower wall of $(\prec)$, then
    $t_{ab}\cdot(\prec)$ is the result of swapping the adjacent pairs $b\precdot a$, $b+n \precdot a+n$, $b+2n\precdot a+2n$, etc. Hence these are only pairs which have a different orientation in $t_{ab}\cdot(\prec)$. It follows that $\tN(t_{ab}\cdot(\prec)) = \tN(\prec)\setminus \{\rind{a,b}\}$, so $t_{ab}\cdot(\prec)$ is covered by $(\prec)$.

    (b)$\implies$(c): If $t_{ab}\cdot (\prec)$ is covered by $(\prec)$, then 
    by the first paragraph there is some lower wall of $(\prec)$ which is not in $\tN(t_{ab}\cdot (\prec))$ or there is an upper wall of $t_{ab}\cdot (\prec)$ which is in $\tN(\prec)$. Applying (a)$\implies$(c) to either case, we get that there is some reflection index $\rind{a',b'}$ so that $\tN(t_{ab}\cdot (\prec)) = \tN(\prec)\setminus\{\rind{a',b'}\}$. Since the pair $(a,b)$ is reversed by $t_{ab}$, we must have $\rind{a',b'}=\rind{a,b}$.

    (c)$\implies$(a): Let $\prec'$ be a TITO so that $a\prec' b$ and $b\prec a$. If there were some $c$ so that $b\prec c \prec a$, then the fact that $a \prec' b$ would imply that $|\tN(\prec)\setminus \tN(\prec')| \geq 2$, by considering cases on whether $c$ is inverted with $b$ or $a$ in $(\prec)$. 

    (d)$\implies$(e): If $\rind{a,a+n}$ is a lower wall of $(\prec)$, then the block containing $a$ must have window $[\underline{a}]$ by \Cref{lem:fundTITO}. Hence, it is waning of size 1.

    (e)$\implies$(f): If the block containing $a$ in $(\prec)$ is waning of size 1, then its window is $[\underline{a}]$. Let $(\prec')$ be the TITO which coincides with $(\prec)$ except that the block containing $a$ has window $[a]$ (i.e. we have reversed the block containing $a$). Then $\tN(\prec')=\tN(\prec)\setminus \{\rind{a,a+kn} \mid k\geq 1\}$. There is no TITO strictly between $(\prec')$ and $(\prec)$ because by \Cref{lem:fundTITO}, there are only two possibilities for the ordering of $\ldots,a-n,a,a+n,\ldots$ by a TITO.

    (f)$\implies$(d): Let $(\prec')$ be the TITO with inversion set $\tN(\prec)\setminus \{\rind{a,a+kn} \mid k\geq 1\}$. Then the block containing $a$ is waxing in $(\prec)$ and waning in $(\prec')$. If there were some $c$ with $a+n \prec c \prec a$, then $c$ would be in the same block as $a$ in $(\prec)$, hence $c$ is also waxing in $(\prec)$ and therefore also waxing in $(\prec')$. Furthermore, we must have $a\prec' c \prec' a+n$, so $c$ and $a$ are also in the same block of $(\prec')$. But then $c$ and $a$ are a waxing and waning element in the same block of $(\prec')$, contradicting \Cref{lem:fundTITO}. 
\end{proof}

If $\rind{a,b}$ is a lower wall of the TITO $(\prec)$, then we write $\mathrm{flip}_{a,b}(\prec)$ for the TITO described in \Cref{thm:wallscovertSn}(c) or (f) (depending on whether $\rind{a,b}$ is real or imaginary, respectively). Then $\mathrm{flip}_{a,b}(\prec)$ is the unique TITO covered by $(\prec)$ so that $\rind{a,b}$ is not an inversion of $\mathrm{flip}_{a,b}(\prec)$.

\subsection{Cyclic arc diagrams and canonical join representations}\label{subsec:cyclicarcs}

Recall that an \emph{arc} (for $S_\infty$) is a tuple $(a,b,L,R)$ so that $a<b$ and $L\sqcup R = \{a+1,a+2,\ldots, b-2,b-1\}$. 

\begin{definition}
    A \newword{wrapped pre-arc} is an equivalence class of arcs under the relation induced by $(a,b,L,R)\sim (a+n,b+n,L+n,R+n)$. The equivalence class containing $(a,b,L,R)$ is denoted $\rind{a,b,L,R}$. A \newword{wrapped arc} is a wrapped pre-arc $\rind{a,b,L,R}$ such that
    \[ \{(a+kn,b+kn,L+kn,R+kn) \mid k\in \ZZ\} \]
    is a non-crossing collection. A wrapped arc is \newword{imaginary} if $a\equiv b \bmod n$; otherwise, it is \newword{real}.
\end{definition}

We will depict a wrapped (pre-)arc $\taa = \rind{a,b,L,R}$ in one of two ways: either by drawing the infinite arc diagram for the collection $\{(a+kn,b+kn,L+kn,R+kn) \mid k \in \ZZ\}$, or by drawing the single arc $(a,b,L,R)$ wrapped clockwise around a circle labeled by $\ZZ/n\ZZ$ as in \Cref{fig:shardarc}.  In the first case, the wrapped pre-arc is a wrapped arc if the arc diagram is non-crossing. In the second case, the wrapped pre-arc is a wrapped arc if it does not cross itself in its interior.  

The wrapped arcs $\rind{a,b,L,R},\rind{a',b',L',R'}$ \newword{cross} if there exist $j,k\in\ZZ$ so that the arcs $(a+jn,b+jn,L+jn,R+jn)$ and $(a'+kn,b'+kn,L'+kn,R'+kn)$ cross. 
A \newword{cyclic non-crossing collection} is a set of wrapped arcs so that no two wrapped arcs cross. Equivalently, $D$ is a cyclic non-crossing collection if
    \[ \bigcup_{\rind{a,b,L,R}\in D} \{(a+kn,b+kn,L+kn,R+kn) \mid k\in \ZZ\} \]
is a non-crossing collection. 

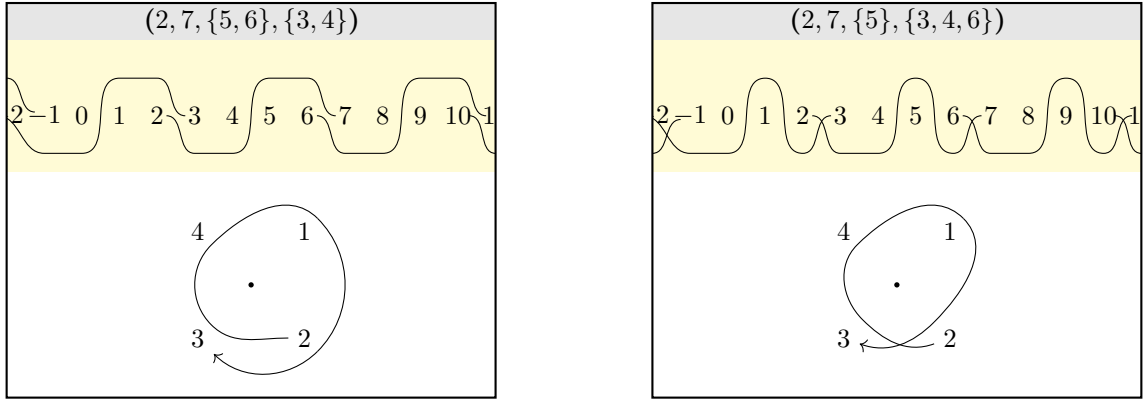
\begin{figure}
    \centering
    \begin{tikzpicture}[scale=.5]
    \fill[gray!20] (-6.5,6) rectangle (6.5,7);
    \fill[yellow!20] (-6.5,2.5) rectangle (6.5,6);
    \draw[thick] (-6.5,-3.5) rectangle (6.5,7);
    \clip (-6.5,-3.5) rectangle (6.5,7);
    \draw node at (0,6.5) {$\rind{2,7,\{5,6\},\{3,4\}}$};
		\begin{scope}[shift={(-3.5,4)}]
			\node (A1) at (0,0) {$1$};
			\node (A2) at (1,0) {$2$};
			\node (A3) at (2,0) {$3$};
			\node (A4) at (3,0) {$4$};
			\node (A5) at (4,0) {$5$};
			\node (A6) at (5,0) {$6$};
			\node (A7) at (6,0) {$7$};
                \node at (7,0) {$8$};
                \node at (8,0) {$9$};
                \node at (9,0) {$10$};
                \node at (10,0) {$11$};
                
                \node at (-1,0) {$0$};
                \node at (-2,0) {$-1$};
                \node at (-3,0) {$-2$};

                \draw (1.25,0) to[out=0,in=180] (2,-1) to[out=0,in=180] (3,-1)
                to[out=0,in=180] (4, 1)
                to[out=0,in=180] (5, 1)
                to[out=0,in=180] (5.75,0);

                \draw (5.25,0) to[out=0,in=180] (6,-1) to[out=0,in=180] (7,-1)
                to[out=0,in=180] (8, 1)
                to[out=0,in=180] (9, 1)
                to[out=0,in=180] (9.75,0);
                \draw (9.3,0) to[out=0,in=180] (10,-1) to[out=0,in=180] (11,-1)
                to[out=0,in=180] (12, 1)
                to[out=0,in=180] (13, 1)
                to[out=0,in=180] (13.75,0);

                \draw (-3.25,0) to[out=0,in=180] (-2,-1) to[out=0,in=180] (-1,-1)
                to[out=0,in=180] (0, 1)
                to[out=0,in=180] (1, 1)
                to[out=0,in=180] (1.75,0);

                \draw (-7.25,0) to[out=0,in=180] (-6,-1) to[out=0,in=180] (-5,-1)
                to[out=0,in=180] (-4, 1)
                to[out=0,in=180] (-3, 1)
                to[out=0,in=180] (-2.25,.1);
		\end{scope}
	
		\begin{scope}[shift={(0,-.5)}]
                \fill (0,0) circle (2pt);
			\node (A1) at (45:2) {$1$};
			\node (A4) at (135:2) {$4$};
			\node (A3) at (225:2) {$3$};
			\node (A2) at (315:2) {$2$};
			\draw (A2) to[out=-180,in=-45] (225:1.5) to[out=135,in=-135] (135:1.5) to[out=45,in=135] (45:2.5);
			\draw (45:2.5) to[out=-45, in=45] (315:2.5);
			\draw[->] (315:2.5) to[out=-135,in=-45] (A3);
		\end{scope}
	\end{tikzpicture}
        \hfill
        \begin{tikzpicture}[scale=.5]
    \fill[gray!20] (-6.5,6) rectangle (6.5,7);
    \fill[yellow!20] (-6.5,2.5) rectangle (6.5,6);
    \draw[thick] (-6.5,-3.5) rectangle (6.5,7);
    \clip (-6.5,-3.5) rectangle (6.5,7);
    \draw node at (0,6.5) {$\rind{2,7,\{5\},\{3,4,6\}}$};
		\begin{scope}[shift={(-3.5,4)}]
			\node (A1) at (0,0) {$1$};
			\node (A2) at (1,0) {$2$};
			\node (A3) at (2,0) {$3$};
			\node (A4) at (3,0) {$4$};
			\node (A5) at (4,0) {$5$};
			\node (A6) at (5,0) {$6$};
			\node (A7) at (6,0) {$7$};
                \node at (7,0) {$8$};
                \node at (8,0) {$9$};
                \node at (9,0) {$10$};
                \node at (10,0) {$11$};
                
                \node at (-1,0) {$0$};
                \node at (-2,0) {$-1$};
                \node at (-3,0) {$-2$};

                \draw (1.25,0) to[out=0,in=180] (2,-1) to[out=0,in=180] (3,-1)
                to[out=0,in=180] (4, 1)
                to[out=0,in=180] (5, -1)
                to[out=0,in=180] (5.75,0);

                \draw (5.25,0) to[out=0,in=180] (6,-1) to[out=0,in=180] (7,-1)
                to[out=0,in=180] (8, 1)
                to[out=0,in=180] (9, -1)
                to[out=0,in=180] (9.75,0);
                \draw (9.3,0) to[out=0,in=180] (10,-1) to[out=0,in=180] (11,-1)
                to[out=0,in=180] (12, 1)
                to[out=0,in=180] (13, -1)
                to[out=0,in=180] (13.75,0);

                \draw (-3.25,0) to[out=0,in=180] (-2,-1) to[out=0,in=180] (-1,-1)
                to[out=0,in=180] (0, 1)
                to[out=0,in=180] (1, -1)
                to[out=0,in=180] (1.75,0);

                \draw (-7.25,0) to[out=0,in=180] (-6,-1) to[out=0,in=180] (-5,-1)
                to[out=0,in=180] (-4, 1)
                to[out=0,in=180] (-3, -1)
                to[out=0,in=180] (-2.25,-.1);
		\end{scope}
	
		\begin{scope}[shift={(0,-.5)}]
                \fill (0,0) circle (2pt);
			\node (A1) at (45:2) {$1$};
			\node (A4) at (135:2) {$4$};
			\node (A3) at (225:2) {$3$};
			\node (A2) at (315:2) {$2$};
                \draw (A2) to[out=200,in=-45] (225:1.3) to[out=135,in=-135] (135:1.5) to[out=45,in=135] (45:2.5);
			\draw (45:2.5) to[out=-45, in=45] (315:1.4);
			\draw[->] (315:1.4) to[out=-135,in=-20] (A3);
		\end{scope}
	\end{tikzpicture}
    \caption{On the left, a wrapped arc drawn in two ways: ``unwrapped'' and ``wrapped''. On the right, a wrapped pre-arc drawn in two ways; note in either way of drawing there is a crossing, so it is not a wrapped arc.}
    \label{fig:enter-label}
\end{figure}

We say the wrapped pre-arc $\taa = \rind{a,b,L,R}$ is a \newword{lower arc} of a TITO $(\prec)$ if the arc $(a,b,L,R)$ is a lower arc of $(\prec)$, viewed as an element of $\Tot$. By \Cref{prop:noncrossingSinfty}, each lower arc is a wrapped arc.  We write $\tLowerArcs(\prec)$ for the set of lower (wrapped) arcs of $(\prec)$. By \Cref{prop:noncrossingSinfty}, $\tLowerArcs(\prec)$ is a non-crossing collection. Recall that $(\prec)$ is \emph{widely generated} if for any $(\prec')<(\prec)$, there is some $(\prec')\leq (\prec^s) \lessdot (\prec)$.

\begin{lemma}\label{lem:titowidetot}
    If $(\prec)$ is a TITO, then $(\prec)$ is widely generated as an element of $\WO(\TTot_n)$ if and only if it is widely generated as an element of $\WO(\Tot)$.
\end{lemma}
\begin{proof}
    First, assume $(\prec)$ is widely generated in $\WO(\TTot_n)$. Consider the total order $(\prec')$ which is the join $\bigvee_{\aa \in \LowerArcs(\prec)}$ in $\WO(\Tot)$. Then $(\prec')$ is translation invariant, since it has an inversion set which is the closure of the translation invariant set $\bigcup_{\aa\in \LowerArcs(\prec)} N(j_\aa)$. Hence $(\prec'),(\prec)$ are both TITOs and $(\prec')\leq (\prec)$. If it were the case that $(\prec')<(\prec)$, then by the fact that $(\prec)$ is widely generated and \Cref{thm:wallscovertSn}, there would be an $\rind{a,b} \in \widetilde{\LowerWalls}(\prec)$ so that $(\prec')\leq \mathrm{flip}_{a,b}(\prec)$. In particular, $(a,b)\not\in N(\prec')$, contradicting the fact given by \Cref{lem:joinofarcsSinfty} that $(a,b)$ is a lower wall of $(\prec')$. Hence $(\prec)=(\prec')$, which is widely generated in $\WO(\Tot)$ by \Cref{lem:joinofarcsSinfty}. 

    Now assume $(\prec)$ is a TITO which is widely generated in $\WO(\Tot)$. If $(\prec')$ is a TITO and $(\prec')<(\prec)$, then there is some $(a,b)\in \LowerWalls(\prec)$ so that $(\prec')\leq (a,b)\cdot (\prec)$. In particular, $\rind{a,b}\not\in \tN(\prec')$. By \Cref{thm:wallscovertSn}, $(\prec')\leq \mathrm{flip}_{a,b}(\prec)$, and $\mathrm{flip}_{a,b}(\prec)$ is covered by $(\prec)$ in $\WO(\TTot_n)$. Hence, $(\prec)$ is widely generated in $\WO(\TTot_n)$. 
\end{proof}

In \Cref{thm:TITOcompact} we will give another characterization of the widely generated TITOs as those TITOs that do not have two consecutive waxing blocks.

\begin{theorem}\label{thm:noncrossingarcsTITO}
If $(\prec)$ is a TITO, then $\tLowerArcs(\prec)$ is a cyclic non-crossing collection. The map $(\prec)\mapsto \tLowerArcs(\prec)$ restricts to a bijection between widely generated elements of $\WO(\TTot_n)$ and cyclic non-crossing collections.
\end{theorem}
\begin{proof}
    This follows from \Cref{lem:titowidetot,prop:noncrossingSinfty}, using the fact that $\bigvee_{\aa\in D}j_\aa$ is translation invariant if $D$ is a non-crossing collection so that $(a,b,L,R)\in D$ implies $(a\pm n,b\pm n, L\pm n, R\pm n)\in D$.
\end{proof}

The \JIs in $\WO(\TTot_n)$ are exactly the widely generated TITOs with exactly one lower wall. Hence \Cref{thm:noncrossingarcsTITO} gives a bijection between $\JIrrc(\WO(\TTot_n))$ and wrapped arcs. Write $j_{\taa}$ for the unique \JI with $\taa$ as a lower arc.

\begin{theorem}\label{thm:canonicaljointSn}
Let $(\prec)\in \WO(\TTot_n)$. Then $(\prec)$ has a canonical join representation if and only if $(\prec)$ is widely generated. In this case, the canonical join representation of $(\prec)$ is 
\[ \bigvee_{\taa\in \tLowerArcs(\prec)} j_{\taa}. \]
\end{theorem}
\begin{proof}
    This follows from \Cref{lem:canonicaljoiniswide,lem:titowidetot,thm:canonicaljoinSinfty}, using that \[j_{\rind{a,b,L,R}} = \bigvee_{k\in \ZZ} j_{(a+kn,b+kn,L+kn,R+kn)}\] (as can be read off from the arc diagram). 
\end{proof}

It is useful to have an explicit description of $j_{\taa}$. We can quickly check that the following elements have $\taa$ as their unique lower arc; furthermore, by comparing with the criterion that will be given in \Cref{thm:TITOcompact}(f), we can also see that they are widely generated. Hence by \Cref{lem:JIwide}, each of these elements is a \JI.



Let $\taa = \rind{a, b, L,R}$ 
be a wrapped arc. 

\textbf{Case 1:} If $b-a < n$, then with $L=\{\ell_1<\cdots < \ell_i\}$ and $R=\{r_1 < \cdots < r_j\}$, $j_{\taa}$ has window notation
\[ j_{\taa} = [\ell_1,\ldots, \ell_i, b,a,r_1,\ldots,r_j,b+1,b+2,\ldots, n+a-1]. \]

\textbf{Case 2:} If $b-a = n$, then with $L=\{\ell_1<\cdots < \ell_i\}$ and $R=\{r_1 < \cdots < r_j\}$,  $j_{\taa}$ has window notation
\[ j_{\taa} =  [\ell_1,\ldots,\ell_i][\underline{a}][r_1,\ldots, r_j]. \]

\textbf{Case 3:} If $b-a > n$ and $a+n \in R$, 
then for an integer $x$ let $\mathcal{C}(x)$ denote the maximum element of $L\cup \ZZ_{<a}$ which is congruent to $x$ mod $n$. 
Let $c_1,\ldots, c_{n-2}$ be the elements of $\{ \mathcal{C}(x) \mid x\not\equiv a,b \mod n  \}$ in sorted order. Then
$j_{\taa}$ has window notation 
\[ j_{\taa} = [c_1,\ldots, c_{n-2}, b, a]. \]

\textbf{Case 4:} If $b-a >n$ and $a+n \in L$, then for an integer $x$ let $L_x$ be the elements of $L$ congruent to $x$ mod $n$ and let $R_x$ be the elements of $R$ congruent to $x$ mod $n$. Let $\cL$ be the set of integers $x$ so that $L_x\neq\varnothing$ and $R_x=\varnothing$, let $\cC$ be the set of integers $x$ so that $L_x,R_x\neq \varnothing$, and let $\cR$ be the set of integers $x$ so that $R_x\neq \varnothing$ and $L_x=\varnothing$. 
Let $\ell_{i_1}< \cdots< \ell_{i_j}$ be the elements of $\{ \max L_x  \mid  x \in \cL \}$, and let $r_{i_1}<\cdots < r_{i_k}$ be the elements of $\{\min R_x \mid x \in \cR\}$.  Let $c_1<\cdots < c_r$ be the elements of $\{ \max L_x \mid x \in \cC\}$. Then $j_{\taa}$ has window notation
\[ j_{\taa} = [\ell_{i_1}, \ldots, \ell_{i_j}][\underline{c_1,\ldots, c_r, b,a}][r_{i_1},\ldots, r_{i_k} ]. \]


\section{Profinite lattices}\label{sec:profinite}
As we will see, an important feature of $\WO(\TTot_n)$ and $\WO(\Tot)$ is that they are both \emph{profinite lattices}, meaning an inverse limit of a directed system of finite lattices. This property will allow us to reduce many questions about $\WO(\TTot_n)$ and $\WO(\Tot)$ to questions about their finite lattice quotients. 

For convenience, we will say that an equivalence relation $\equiv$ on $L$ is a \newword{cofinite congruence} if it is a complete lattice congruence and the quotient lattice $L/{\equiv}$ is finite.

\begin{lemma}\label{lem:refinementcof}
    If $L$ is a complete lattice and $\equiv_1,\ldots, \equiv_k$ is a finite list of cofinite congruences, then there is a cofinite congruence simultaneously refining $\equiv_1,\ldots,\equiv_k$.
\end{lemma}
\begin{proof}
    The quotient homomorphisms $L\to L/{\equiv}_i$ induce a complete lattice homomorphism $\eta:L \to L/{\equiv_1} \times \cdots \times L/{\equiv}_k$. Let $L'$ be the image of $\eta$. By \Cref{prop:congruence}, there is a complete lattice congruence $\equiv$ on $L$ so that $x\equiv y$ if and only if $\eta(x)=\eta(y)$. Then $\equiv$ is a common refinement of $\equiv_1,\ldots,\equiv_k$. Furthermore, $\equiv$ is cofinite since $|L/{\equiv}| \leq |L/{\equiv}_1|\cdot \cdots \cdot |L/{\equiv}_k| < \infty$. 
\end{proof}

\begin{definition}\label{def:profinite}
    A complete lattice $L$ is called \newword{profinite} if for any $x,y\in L$, if $x\neq y$ then there is a cofinite lattice congruence $\equiv$ on $L$ so that $x\not\equiv y$.
\end{definition}

We get the following corollary of \Cref{lem:refinementcof}.
\begin{corollary}\label{lem:cofinitelist}
    If $L$ is a profinite lattice, then for any finite list $x_1\neq y_1, \ldots, x_k\neq y_k$ there is some cofinite congruence $\equiv$ on $L$ so that $x_1 \not\equiv y_1, \ldots,$ and $x_k\not\equiv y_k$.
\end{corollary}

\begin{lemma}\label{lem:sublatticeprofinite}
If $L'$ is a complete sublattice of a profinite lattice $L$, then $L'$ is profinite.
\end{lemma}
\begin{proof}
Pick $x,y \in L'$ with $x\neq y$. Since $L$ is profinite, there is some cofinite congruence $\equiv$ on $L$ so that $x\not\equiv y$. The restriction of $\equiv$ to $L'$ is a complete lattice congruence on $L'$ by \Cref{prop:sublatticecongruence}. Furthermore, the restriction of $\equiv$ to $L'$ is cofinite, for instance because $L'/{\equiv}$ can be identified with the image of the composite map $L'\hookrightarrow L \to L/{\equiv}$ so that $|L'/{\equiv}|\leq |L/{\equiv}|$. Hence $L'$ is profinite.
\end{proof}

\begin{aside}
\begin{remark}
    It is possible that a complete lattice quotient of a profinite lattice is no longer profinite. See \cite[Example 2.13]{Iyama2018} for an example.
\end{remark}
\end{aside}

We will focus on a particular family of congruences with an explicit description. We say the family $\{\equiv_i\}_{i\in I}$ \newword{detects equality} if, whenever $x\equiv_i 
 y$ for all $i\in I$, then $x=y$.
Then a complete lattice $L$ is profinite if it has a collection of cofinite lattice congruences which detects equality.

Write $T_{A,B} \coloneqq \{(a,b) \mid A\leq a < b \leq B\}$. For $A,B\in \ZZ$, define a congruence $\equiv_{A,B}$ on $\Tot$ by declaring $(\prec_1) \equiv_{A,B} (\prec_2)$ if and only if $N(\prec_1) \cap T_{A,B} = N(\prec_2) \cap T_{A,B}$. Equivalently, $(\prec_1)\equiv_{A,B} (\prec_2)$ if and only if $(\prec_1)$ and $(\prec_2)$ agree on the ordering of $\{A,\ldots, B\}$. Observe that $\equiv_{A,B}$ is cofinite, since $|L/{\equiv}_{A,B}| = (B-A+1)!$.

\begin{theorem}
$\WO(\Tot)$ and $\WO(\TTot_n)$ are profinite lattices.
\end{theorem}
\begin{proof}
We will show $\WO(\Tot)$ is a profinite lattice; by \Cref{lem:sublatticeprofinite,thm:latticetSn}, it will follow that $\WO(\TTot_n)$ is a profinite lattice as well. We claim that $\{\equiv_{A,B}\}_{A<B}$ detects equality. Let $(\prec_1),(\prec_2)$ be distinct total orders of the integers. Then there must be integers $a,b\in \ZZ$ so that $a \prec_1 b$ and $b \prec_2 a$. Write $A=\min\{a,b\}$ and $B=\max\{a,b\}$. Then one of the sets $N(\prec_1)\cap T_{A,B},~ N(\prec_2)\cap T_{A,B}$ contains the pair $(A,B)$ and the other does not. Hence $(\prec_1)\not\equiv_{A,B} (\prec_2)$.
\end{proof}

\subsection{Compact elements in profinite lattices}
In this subsection, we will study the notion of \emph{compactness} in a profinite lattice.
\begin{definition}
Let $L$ be a complete lattice. An element $x\in L$ is \newword{compact} if, for any subset $X\subseteq L$ so that $x\leq \bigvee X$, there is some finite subset $X'\subseteq X$ so that $x\leq \bigvee X'$. If every element of $L$ is the join of a set of compact elements, then $L$ is an \newword{algebraic lattice}.
\end{definition}

Recall from \Cref{prop:congruence} that each complete lattice congruence $\equiv$ on a complete lattice $L$ comes with order-preserving maps $\pi^\downarrow_\equiv, \pi^\uparrow_\equiv : L\to L$ sending $x$ to the minimal (respectively, maximal) element of the equivalence class $[x]_\equiv$.

\begin{lemma}\label{lem:cofinitecovers}
    Let $L$ be a complete lattice and $\equiv$ a cofinite congruence on $L$. If $x$ is any element of the quotient lattice $L/{\equiv}$, then $\pi^{\downarrow}_\equiv(x)$ is widely generated. Furthermore, the lower covers of $\pi^{\downarrow}_\equiv(x)$ are exactly the elements $\{ \pi^{\uparrow}_\equiv(y) \mid y \lessdot x \}$.
\end{lemma}
\begin{proof}
    Assume $\equiv$ is a cofinite congruence on $L$ and $x\in L/{\equiv}$. Take some $y<\pi^{\downarrow}_\equiv(x)$ and consider the set $Y=\{y' \in L \mid y\leq y' < \pi^{\downarrow}_\equiv(x)\}$. We will show that $Y$ satisfies the hypothesis of Zorn's lemma, so that a maximal element of $Y$ gives a lower cover $y'\lessdot \pi^{\downarrow}_\equiv(x)$ with $y\leq y'$. Indeed, consider a totally ordered subset $\cC\subseteq Y$. Then we claim $\bigvee \cC < \pi^{\downarrow}_\equiv(x)$. To see this, let $\eta: L\to L/{\equiv}$ be the quotient map. Then $\eta(\bigvee \cC) = \bigvee \eta(\cC) $. Note that $\eta(\cC)$ is a chain in the finite lattice $L/{\equiv}$, so $\bigvee \eta(\cC)$ is the maximal element of $\eta(\cC)$. In particular, $\bigvee \eta(\cC) < x$, which implies that $\bigvee \cC < \pi^\downarrow_\equiv(x)$. Hence $\pi^{\downarrow}_\equiv(x)$ is widely generated. Furthermore, the lower covers of $\pi^{\downarrow}_\equiv(x)$ are in bijection with the lower covers of $x$ in $L/{\equiv}$, because if $y\lessdot \pi^{\downarrow}_\equiv(x)$ then $\eta(y)\lessdot x$, and if $y_1,y_2\lessdot \pi^{\downarrow}_\equiv(x)$ and $\eta(y_1)=\eta(y_2)$, then $y_1\vee y_2 < \pi^{\downarrow}_\equiv(x)$, so $y_1=y_2$.
\end{proof}

We say that an element $x$ of a complete lattice $L$ is \newword{finitely generated} if there is some cofinite congruence $\equiv$ on $L$ so that $x = \pi^\downarrow_\equiv(x)$. 

\begin{proposition}\label{prop:compactprofinite}
    The following are equivalent, for an element $x$ of a profinite lattice $L$:
    \begin{itemize}
        \item[(a)] $x$ is compact;
        \item[(b)] $x$ is finitely generated;
        \item[(c)] $x$ is widely generated and $x$ covers finitely many elements;
        \item[(d)] $x$ is the join of finitely many complete join-irreducibles.
    \end{itemize}
\end{proposition}
\begin{proof}
    Let $\{\equiv_i\}_{i\in I}$ be a collection of cofinite congruences on $L$ which detects equality. 
    
    (a)$\implies$(b): Observe that $x = \bigvee_{i\in I} \pi^{\downarrow}_{\equiv_i}(x)$, since otherwise there is an $i\in I$ so that $x \not\equiv_i \bigvee_{i\in I} \pi^{\downarrow}_{\equiv_i}(x)$, which contradicts the fact that $x\geq \bigvee_{i\in I} \pi^{\downarrow}_{\equiv_i}(x)$ and $x\equiv_i \pi^{\downarrow}_{\equiv_i}(x)$. Since $x$ is compact, this means there is a finite subset $I'\subseteq I$ so that $x = \bigvee_{i\in I'} \pi^{\downarrow}_{\equiv_i}(x)$.  Let $\equiv$ be a cofinite common refinement of $\{\equiv_i\}_{i\in I'}$. Then $x=\pi^{\downarrow}_\equiv(x)$.

    (b)$\implies$(c): This is a direct consequence of \Cref{lem:cofinitecovers}, since the lower covers of $x$ are in bijection with the lower covers of $[x]_\equiv$ in $L/{\equiv}$, which is finite.

    (c)$\implies$(d): Assume $x$ is widely generated and $y_1,\ldots,y_k\lessdot x$ are the lower covers of $x$. By \Cref{lem:cofinitelist}, there is a cofinite congruence $\equiv$ so that $y_1,\ldots,y_k\not\equiv x$. Write $\eta : L\to L/{\equiv}$ for the quotient map; our hypothesis on $\equiv$ is equivalent to $\eta(x)> \eta(y_1),\ldots,\eta(y_k)$. Then $\eta(x)$ is a join of join-irreducibles $j_1,\ldots, j_r \in L/{\equiv}$. Each $\pi^{\downarrow}_\equiv(j_i)$ is a complete join-irreducible, since by \Cref{lem:cofinitecovers} it is widely generated and has a unique lower cover, which implies it is a \JI by \Cref{lem:JIwide}. We claim $x = \pi^{\downarrow}_\equiv(j_1) \vee \cdots \vee \pi^{\downarrow}_\equiv(j_r)$. 
    Indeed, if not then $\pi^{\downarrow}_\equiv(j_1) \vee \cdots \vee \pi^{\downarrow}_\equiv(j_r) < x$, so by using the fact that $x$ is widely generated, we would have that $\pi^{\downarrow}_\equiv(j_1) \vee \cdots \vee \pi^{\downarrow}_\equiv(j_r) \leq y_i$ for some $i$. But then applying $\eta$ we get $j_1\vee \cdots \vee j_r \leq \eta(y_i)$, which contradicts the fact that $\eta(x)> \eta(y_i)$. The claim follows.

    (d)$\implies$(a): Compactness is preserved under taking finitely many joins, so it is enough to check that \JIs are compact. Assume $j$ is a \JI in $L$ covering the element $j_*$. Let $\equiv$ be a cofinite congruence with $j_*\not\equiv j$, and write $\eta:L\to L/{\equiv}$ for the lattice quotient. Given a collection $X\subseteq L$ so that $j\leq \bigvee X$, there is some finite subset $X'\subseteq X$ so that $\eta(j) \leq \bigvee \eta(X')$. Then also $\pi^{\downarrow}_\equiv(j) \leq  \pi^{\downarrow}_\equiv(\bigvee \eta(X')) = \bigvee \pi^{\downarrow}_\equiv(\eta(X')) = \bigvee X'$\footnote{Note that $\pi^{\downarrow}_\equiv$ does preserve joins, but does not generally preserve meets.}. Because $j$ is a \JI, by \Cref{lem:JIwide} if $\pi^{\downarrow}_\equiv(j)$ were strictly less than $j$ then we would have $\pi^{\downarrow}_\equiv(j)\leq j_*$, contradicting the fact that $j_*\not\equiv j$.  So we must have $j=\pi^{\downarrow}_\equiv(j)$ and we are done. 
    
\end{proof}

\begin{corollary}\label{cor:algebraic}
    Every profinite lattice is algebraic. In particular, $\WO(\Tot)$ and $\WO(\TTot_n)$ are algebraic lattices.
\end{corollary}
\begin{proof}
    In a profinite lattice, any element $x$ is the join of the finitely generated elements $\{ \pi^{\downarrow}_\equiv(x) \mid {\equiv} \text{ is a cofinite congruence on } L\}$. By \Cref{prop:compactprofinite}, finitely generated elements are compact, so $x$ is the join of compact elements.
\end{proof}

\begin{theorem}\label{thm:totcompact}
    The following are equivalent, for a total order $(\prec)\in \WO(\Tot)$:
    \begin{itemize}
        \item[(a)] $(\prec)$ is compact;
        \item[(b)] $(\prec)$ is finitely generated;
        \item[(c)] $(\prec)$ is widely generated and covers finitely many total orders;
        \item[(d)] $(\prec)$ is the join of finitely many complete join-irreducibles;
        \item[(e)] The inversion set $N(\prec)$ is finite;
        \item[(f)] There is some $\pi \in S_\infty$ so that $(\prec) = (\prec_\pi)$.
    \end{itemize}
\end{theorem}
\begin{proof}
    The equivalence of (a-d) is \Cref{prop:compactprofinite}. The equivalence of (e) and (f) is straightforward. The equivalence of (d) and (e) follows since the description of \JIs in \Cref{subsec:canjoinSinfty} implies that every completely join irreducible total order has a finite inversion set.
\end{proof}

\begin{theorem}\label{thm:TITOcompact}
    The following are equivalent, for a TITO $(\prec) \in \WO(\TTot_n)$:
    \begin{itemize}
        \item[(a)] $(\prec)$ is compact;
        \item[(b)] $(\prec)$ is finitely generated;
        \item[(c)] $(\prec)$ is widely generated;
        \item[(d)] $(\prec)$ is the join of finitely many complete join-irreducibles;
        \item[(e)] There is a finite subset $R \subseteq \tT$ so that $\tN(\prec) = \overline{R}$;
        \item[(f)] $(\prec)$ does not have two consecutive waxing blocks.
    \end{itemize}
\end{theorem}
\begin{proof}
    The equivalence of (a-d) is \Cref{prop:compactprofinite}, except that we need to check that every widely generated TITO has finitely many lower covers. Indeed, by \Cref{thm:wallscovertSn}, the lower covers correspond to lower walls. A block of size $k$ contains at most $k$ lower walls, so any TITO has at most $n$ lower walls.

    (d)$\implies$(e): It is enough to show, for each wrapped arc $\taa = (a,b,L,R)$, that the \JI $j_\taa$ has the property that $\tN(j_\taa) = \overline{R}$ for a finite set $R\subseteq \tT$. Indeed, we claim that $\tN(j_\taa)$ is the closure of the finite set $\tA_0 = \tN(j_\taa)\cap \{\rind{x,y} \mid y-x \leq b-a \}$. One can check this directly using the list in \Cref{subsec:cyclicarcs}. Alternatively, note that $\rind{x,y}\in \overline{\tA}_0$ if and only if $(x,y)\in \overline{A}_0$, where $A_0\coloneqq \{ (x,y) \mid \rind{x,y}\in \tA_0 \}$. Furthermore, $A_0$ is exactly $\{(x,y)\mid x<y,~y-x\leq b-a \} \cap \bigcup_{k\in\ZZ} N(j_{a+kn,b+kn,L+kn,R+kn})$. It is quick to check that $N(j_{(a,b,L,R)}) = \overline{N(j_{(a,b,L,R)})\cap \{(x,y)\mid x<y,~y-x\leq b-a \}}$. Since $j_\taa$ is the join of $\{ j_{(a+kn,b+kn,L+kn,R+kn)} \mid k\in\ZZ\}$ in $\WO(\Tot)$, this finishes the claim. 

    (e)$\implies$(a): Let $(\prec)$ be a TITO and $R\subseteq \tT$ a finite set so that $\tN(\prec)=\overline{R}$. If $X$ is a subset of $\WO(\TTot_n)$ so that $(\prec) \leq \bigvee X$, then in particular $R \subseteq \overline{\bigcup_{x\in X}\tN(x) }$. Since $R$ is finite, there is a finite subset $X'\subseteq X$ so that $R\subseteq \overline{\bigcup_{x\in X}\tN(x)}$. But then also $\tN(\prec) = \overline{R} \subseteq \overline{\bigcup_{x\in X}\tN(x)} $, so we are done.

    (c)$\implies$(f): Assume $(\prec)$ has two consecutive waxing blocks. Say the windows for the two blocks are $[x_1,\ldots,x_j][y_1,\ldots, y_k]$. Pick a $K\gg 0$ large enough so that $\max\{x_1,\ldots,x_j\} > \min\{ y_1-Kn, \ldots, y_k-Kn\}$. Then replacing the two windows $[x_1,\ldots,x_j][y_1,\ldots, y_k]$ with the single window $[x_1,\ldots,x_j,y_1-Kn,\ldots,y_k-Kn]$ gives a new TITO $(\prec')$ so that $(\prec')<(\prec)$. Furthermore, every lower wall of $(\prec)$ is still an inversion of $(\prec')$, so by \Cref{thm:wallscovertSn} there is no lower cover of $(\prec)$ bigger than $(\prec')$. Hence $(\prec)$ is not widely generated.

    (f)$\implies$(b): Let $(\prec)$ be a TITO and fix a window notation for $(\prec)$. Pick $A,B\in \ZZ$ so that the interval $[A,B]_{<}$ includes all $n$ integers appearing in the windows of $(\prec)$. Then for each block (say, of size $k$), the interval $[A,B+n]_{<}$ includes $k+1$ consecutive integers from that block. Hence if we restrict $(\prec)$ to the integers in $[A,B+n]_{<}$, then using \Cref{lem:fundTITO} we can recover the divisions between waxing and waning blocks. The only ambiguity is when adjacent blocks are waxing or adjacent blocks are waning, in which case multiple TITOs give the same ordering of $[A,B+n]_<$. The unique minimal TITO inducing that order is $\pi^{\downarrow}_{\equiv_{A,B+n}}(\prec)$ (in particular, it exists). The unique TITO inducing that order which also minimizes the number of waxing blocks and maximizes the number of waning blocks is $(\prec)$. But then $(\prec)$ must be $\pi^{\downarrow}_{\equiv_{A,B+n}}(\prec)$, since if we fix a choice of waxing or waning for each integer $x$ and only consider TITOs respecting that choice, then going down in $\WO(\TTot_n)$ can only decrease the number of waxing blocks and increase the number of waning blocks.
\end{proof}
The equivalence between (e) and (f) (or rather, its analog for extended weak order), was first shown by Weijia Wang in \cite{Wang2019}.

\begin{aside}
\begin{remark}
There is a different approach to show that (b) and (e) are equivalent. Imagine we have an exhaustive filtration $\tT_1\subseteq \tT_2 \subseteq \tT_3 \subseteq \cdots \subseteq \tT$ so that each $\tT_i$ is finite and the map $(\prec) \mapsto \tN(\prec) \cap \tT_i$ is a complete lattice homomorphism onto its image (ordered by containment). Let $\equiv_i$ be the induced cofinite congruences. If it is also the case that $\overline{\tN(\prec)\cap \tT_i}$ is the inversion set of a TITO for all TITOs $(\prec)$, then it must be that $\tN(\pi^{\downarrow}_{\equiv_i}(\prec))$ is the closure of the finite set $\tN(\prec) \cap \tT_i$. Hence we deduce the equivalence of (b) and (e). Indeed, there is such a filtration if we take $\tT_i = \{ \rind{a,b} \mid b-a \leq i\}$. A related filtration is used in \cite{Barkley2023} to study extended weak order, and has useful analogs in other affine Coxeter groups.
\end{remark}
\end{aside}

\subsection{Semidistributivity in profinite lattices}

In this section we show that semidistributivity in profinite lattices behaves similarly to finite lattices. First, a lemma.

\begin{lemma}\label{lem:joinrepcharacter}
    Let $L$ be a complete lattice. Then $x\in L$ has a canonical join representation if and only if $x$ is widely generated and, for each $y\lessdot x$, the set $\{z \in L \mid z \vee y = x\}$ has a minimum\footnotemark element $j_y$. In that case, the canonical join representation is $x = \bigvee_{y\lessdot x} j_y$.
\end{lemma}
\footnotetext{A set $X$ has a \emph{minimum} if the greatest lower bound of $X$ is in $X$. For infinite lattices, this is stronger than $X$ having a unique minimal element.}
\begin{proof}
    First, if $x$ has a canonical join representation then $x$ is widely generated by \Cref{lem:canonicaljoiniswide}. If $x = \bigvee X$ is the canonical join representation and $j\in X$, then there is a unique $y\lessdot x$ so that $j\not\leq y$. This sets up a bijection between the elements of $X$ and the lower covers of $x$. Furthermore, if $z\in L$ is any element so that $z\vee y= x$, then $x=z \vee \bigvee_{\substack{j'\in X\\ j'\neq j}}j'$ so $j\leq z$ by minimality of the canonical join representation. 

    Conversely, assume $x$ is widely generated and for each lower cover $y\lessdot x$, there is a minimum element $j_y$ of $\{z\in L \mid z \vee y =x\}$. Then $\bigvee_{y\lessdot x} j_y =x$, or else we would contradict the fact that $x$ is widely generated. Note that for two distinct lower covers $y_1,y_2\lessdot x$, we have $j_{y_1}\leq y_2$. Hence the join representation $x=\bigvee_{y\lessdot x}j_y$ is irredundant. Now let $x= \bigvee V$ be any join representation and let $y\lessdot x$. There must be some element $v\in V$ so that $v\not\leq y$. But then $j_y \leq v$ by its definition. Hence $x=\bigvee_{y\lessdot x}j_y$ is the canonical join representation.
\end{proof}

Recall the definitions of join semidistributive and completely join semidistributive from \Cref{def:semidistributive}.

\begin{proposition}\label{prop:profinitesemidist}
    Let $L$ be a profinite lattice. Then the following are equivalent:
    \begin{itemize}
        \item[(a)] $L$ is join semidistributive;
        \item[(b)] $L$ is completely join semidistributive;
        \item[(c)] For every cofinite congruence $\equiv$ on $L$, the quotient lattice $L/{\equiv}$ is join semidistributive;
        \item[(d)] Every widely generated element of $L$ has a canonical join representation;
        \item[(e)] Every finitely generated element of $L$ has a canonical join representation.
    \end{itemize}
\end{proposition}
\begin{proof}
    (a)$\implies$(c): By \cite[Exercise 9.49]{Reading2016a}, any complete lattice quotient of a semidistributive lattice is semidistributive. 

    (c)$\implies$(b): Let $X\subseteq L$ and $y,z\in L$. If $x\vee y= z$ for all $x\in X$, then for any congruence $\equiv$ we have $x\vee y \equiv z$ for all $x\in X$. If $\equiv$ is a cofinite congruence, then semidistributivity of the finite lattice $L/{\equiv}$ is equivalent to its complete semidistributivity, so we deduce that $(\bigwedge X)\vee y \equiv z$. Since this is true for all cofinite congruences and $L$ is profinite, we conclude that $(\bigwedge X)\vee y = z$. 

    (b)$\implies$(a): Obvious.

    (b)$\implies$(d): Let $x\in L$ be widely generated. By \Cref{lem:joinrepcharacter}, it is enough to show that, for each $y\lessdot x$, the set $Z=\{z\in L \mid z\vee y = x\}$ has a minimum element. Applying complete join semidistributivity, we find that $(\bigwedge Z) \vee y = x$, so the claim follows.
    

    (d)$\implies$(e): Follows from \Cref{prop:compactprofinite}.

    (e)$\implies$(c): Let $\equiv$ be a cofinite congruence and $\eta:L\to L/{\equiv}$ the quotient map. Then for any $x\in L$ so that $\pi^{\downarrow}_\equiv(x)=x$, the finitely generated element $x$ has a canonical join representation. Hence for each $y\lessdot x$, the set $\{z\in L \mid z\vee y = x\}$ has a minimum element $j_y$ by \Cref{lem:joinrepcharacter}. Note that for any $z\in L$ so that $z\leq x$, we have $z\vee y = x$ if and only if $\eta(z) \vee \eta(y) = \eta(x)$. As a result, $\eta(j_y)$ is the minimum element of $\{\eta(z) \mid z\in L,~\eta(z) \vee \eta(y) = \eta(x)\}$. 
    By \Cref{lem:joinrepcharacter}, $L/{\equiv}$ has canonical join representations, so by \Cref{prop:finitecanonicaljoin} it is join semidistributive.

\end{proof}

\begin{aside}
    \begin{remark}
        Without the assumption of profiniteness, it is still the case that every completely semidistributive lattice satisfies (a-e).
    \end{remark}
\end{aside}

We can now prove two-thirds of the theorem announced in the introduction.
\begin{theorem}\label{thm:semidistr}
    The lattices $\WO(\Tot)$ and $\WO(\TTot_n)$ are completely semidistributive.
\end{theorem}
\begin{proof}
    By \Cref{thm:canonicaljoinSinfty,thm:canonicaljointSn}, every widely generated element of the two lattices has a canonical join representation. Thus \Cref{prop:profinitesemidist} implies the lattices are completely join semidistributive. The dual versions of \Cref{thm:canonicaljoinSinfty,thm:canonicaljointSn,prop:profinitesemidist} imply that the lattices are also completely meet semidistributive.  
\end{proof}



\subsection{Extended weak order}
The following was shown in \cite{Barkley2022}, and is the original motivation for studying $\WO(\TTot_n)$. We use the notation for the root system of $\tS_n$ from \Cref{tSnroots}.

\begin{theorem}\label{thm:TITOtoDyer}
   The map 
   \[ (\prec) \mapsto \{\talpha_{ab} \mid \rind{a,b} \in \tN(\prec), ~ a\not\equiv b \bmod n\} \]
   is a surjective complete lattice homomorphism from $\WO(\TTot_n)\to \Dyer(\tS_n)$. 
\end{theorem}

The property of being completely semidistributive is preserved by complete lattice quotients \cite[Exercise 9.49]{Reading2016a}. Using \Cref{thm:semidistr}, we deduce another piece of our main theorem.

\begin{corollary}
    $\Dyer(\tS_n)$ is completely semidistributive.
\end{corollary}

We conclude with the last piece of \Cref{thm:intro}.

\begin{theorem}
    $\Dyer(\tS_n)$ is a profinite lattice.
\end{theorem}
\begin{proof}
We view $\Dyer(\tS_n)$ as the quotient lattice of $\WO(\TTot_n)$ by the congruence which puts $(\prec_1)\equiv(\prec_2)$ if the real reflection indices in $\tN(\prec_1)$ and $\tN(\prec_2)$ are the same. Two TITOs are equivalent if and only if they differ by reversing blocks of size 1. Each congruence on $\WO(\TTot_n)$ coarsening $\equiv$ corresponds to a congruence of $\Dyer(\tS_n)$. Let $\equiv'_{A,B}$ be the finest lattice congruence coarsening both $\equiv$ and the congruence $\equiv_{A,B}$ introduced above. Then $\equiv'_{A,B}$ is cofinite. We wish to check that the congruences $\equiv'_{A,B}$ detect equality in $\Dyer(\tS_n)$. Since equivalence classes for $\equiv_{A,B}$ correspond to TITOs which order $[A,B]_{<}$ the same way, equivalence classes for $\equiv_{A,B}'$ correspond to TITOs which order $[A,B]_{<}$ the same way up to reversing consecutive subsequences in the same residue class modulo $n$. Given two TITOs $(\prec_1),(\prec_2)$ with $(\prec_1)\not\equiv (\prec_2)$, there exists integers $a,b$ with $a\not\equiv b\bmod n$ so that $a\prec_1 b$ and $b\prec_2 a$. Setting $A=\min\{a,b\}$ and $B=\max\{a,b\}$, we find that $(\prec_1)\not\equiv_{A,B}' (\prec_2)$, since reversing consecutive subsequences in a single residue class can never change the order of two elements in different residue classes. Hence, the family $(\equiv_{A,B}')_{A<B}$ detects equality on $\Dyer(\tS_n)$.
\end{proof}

\section{Acknowledgments}
I was partially supported by NSF Grant DMS-1854512. I would like to thank Colin Defant, Arnau Padrol, Vincent Pilaud, Nathan Reading, David Speyer, Hugh Thomas, and Lauren Williams for helpful conversations related to this work.


\bibliographystyle{plain}
\bibliography{main}

\end{document}